\documentclass[11pt,a4paper,final]{amsart}

\usepackage{amsmath}
\usepackage{paralist}
\usepackage{graphics}
\usepackage{graphicx}
\usepackage{epstopdf}
\usepackage{amssymb}
\usepackage{amsthm}
\usepackage{color}
\usepackage{multirow}
\usepackage{booktabs}
\usepackage{bm}
\usepackage{mathrsfs}

\usepackage[utf8]{inputenc}
\usepackage{url}

\newtheorem{theorem}{Theorem}
\newtheorem{corollary}{Corollary}

\newtheorem{proposition}{Proposition}

\newtheorem{definition}{Definition}
\newtheorem{remark}{Remark}
\newtheorem{example}{Example}

\newcommand{\R}{\mathbb{R}}

\title[On the Moore-Penrose pseudo-inversion of block matrices]{On the Moore-Penrose pseudo-inversion of block symmetric matrices and its application in the graph theory}

\author{So\v{n}a Pavl\'{\i}kov\'a${}^{1}$}
\address{${}^{1}$ Alexander Dub\v{c}ek University of Tren\v{c}\'{\i}n, Slovakia}

\author{Daniel \v{S}ev\v{c}ovi\v{c}${}^{2}$}

\address{${}^{2}$ Department of Applied Mathematics and Statistics, Faculty of Mathematics Physics and Informatics, Comenius University, Mlynsk\'a dolina, 842 48, Bratislava, Slovakia. Corresponding author: {\tt sevcovic@fmph.uniba.sk} }

\thanks{Acknowledgments: Support of the Slovak Research and Development Agency under the projects APVV-19-0308 (SP), and APVV-20-0311 (DS) is kindly acknowledged.}

\begin{document}

\maketitle

\begin{abstract}
The purpose of this paper is to analyze the Moore-Penrose pseudo-inversion of symmetric real matrices with application in the graph theory. We introduce a novel concept of positively and negatively pseudo-inverse matrices and graphs. We also give sufficient conditions on the elements of a block symmetric matrix yielding an explicit form of its Moore-Penrose pseudo-inversion. Using the explicit form of the pseudo-inverse matrix we can construct pseudo-inverse graphs for a class of graphs which are constructed from the original graph by adding pendent vertices or pendant paths. 



\end{abstract}

\section{Introduction}

In this paper we investigate the Moore-Penrose pseudo-inversion of block symmetric matrices with application in the graph theory. We present sufficient conditions on the elements of a block matrix yielding an explicit block matrix form of its Moore-Penrose inversion. We also introduce a novel concept of positively and negatively pseudo-invertible matrix. In the context of the graph theory we extend the concept of positive and negative invertibility of a graph due to Godsil \cite{Godsil1985}, Pavl\'\i kov\'a and \v{S}ev\v{c}ovi\v{c} \cite {Pavlikova2016} to the case when its adjacency matrix is not invertible.  

In the past decades various ways of introducing inverses of graphs have been proposed. They are based on inversion of the adjacency matrix. Except of the trivial case when the graph is a union of isolated edges, the inverse matrix need not define a graph again because it may contain negative entries (cf. \cite{Har}).  A successful approach how to overcome this difficulty was proposed by Godsil \cite{Godsil1985} who defined a graph to be invertible if the inverse of its non-singular adjacency matrix is diagonally similar (cf. \cite{Zas}) to a nonnegative integral matrix representing the adjacency matrix of the inverse graph in which positive labels determine edge multiplicities. In \cite {Pavlikova2016},  Pavl\'\i kov\'a and \v{S}ev\v{c}ovi\v{c} extended this notion to a wider class of graphs by introducing the concept of negative invertibility of a graph. Both positively and negatively invertible graphs have the appealing property that inverting an inverse graph gives back the original graph. For a survey of results and other approaches to graph inverse, we recommend \cite{McMc}. Godsil's ideas have been further developed in several ways. Akbari and  Kirkland \cite{KirklandAkb2007}, Kirkland and Tifenbach \cite{KirklandTif2009},  and Bapat and Ghorbani \cite{Bapat} studied inverses of edge-labeled graphs with labels in a ring, Ye \emph{et al.} \cite{Ye} considered connections of graph inverses with median eigenvalues, and Pavl\'{\i}kov\'a \cite{Pavlikova2015} developed constructive methods for generating invertible graphs by edge overlapping. A large number of related results, including a unifying approach to inverting graphs, were proposed in a recent survey paper by McLeman and McNicholas \cite{McMc}, with emphasis on inverses of bipartite graphs and diagonal similarity to nonnegative matrices.

In theoretical chemistry, biology, or statistics, spectral indices and properties of graphs representing structure of chemical molecules or transition diagrams for finite Markov chains play an important role (cf. Cvetkovi\'c \cite{Cvetkovic1988,Cvetkovic2004}, Brouwer and Haemers \cite{Brouwer2012} and references therein). Various graph energies and indices have been proposed and analyzed. For instance, the sum of absolute values of eigenvalues is referred to as the matching energy index (cf. Chen and Liu \cite{Lin2016}), the maximum of the absolute values of the least positive and largest negative eigenvalue is known as the HOMO-LUMO index (see Mohar \cite{Mohar2013,Mohar2015}, Li \emph{et al.} \cite{Li2013}, Jakli\'c  \emph{et al.} \cite{Jaklic2012}, Fowler \emph{et al.} \cite{Fowler2010}), their difference is the  HOMO-LUMO separation gap (cf.  Gutman and Rouvray \cite{Gutman1979}, Li \emph{et al.} \cite{Li2013}, Zhang and An \cite{Zhang2002}, Fowler  \emph{et al.} \cite{Fowler2001}). It turns out that properties of maximal and minimal eigenvalues of the inverse graph can be used in order to analyze the least positive and largest negative eigenvalues of an adjacency matrix. 

Recently, McDonald, Raju, and Sivakumar investigated Moore-Penrose inverses of matrices associated with certain graph classes \cite{McDonald} whose adjacency matrices are singular. They derived formulae for pseudo-inverses of matrices that are associated with a class of digraphs obtained from stars. This new class contains both bipartite and non-bipartite graphs. A blockwise representation of the (pseudo)inverse of the adjacency matrix of the Dutch windmill graph is also presented in \cite{McDonald}. In \cite{PaKri} Pavl\'{\i}kov\'{a}, and Krivo\v{n}\'{a}kov\'{a} gave characterization of pseudo-inverses of cycles. In the context of the Laplacian of connections of a graph, Hessert and Mallik \cite{HesMal} and Gago \cite{Gago} studied Moore-Penrose inverses of the signless Laplacian and edge-Laplacian of graphs or networks. They found combinatorial formulae of the Moore-Penrose inverses for trees and odd unicyclic graphs. In \cite{Pavlikova2019-Australasian} Pavl\'{\i}kov\'{a} and \v{S}ir\'a\v{n} constructed a pseudo-inverse of a weighted tree, based solely on considering maximum matchings and alternating paths.

Recall that the complete list due to McKay of simple connected graphs is available at \url{http://users.cecs.anu.edu.au/~bdm/data/graphs.html}. In Section 3 (see Table~\ref{tab-1}) we present our computations of total numbers of simple connected graphs based on the McKay's list together with a survey of integrally positive/negative invertible graphs (see Pavl\'\i kov\'a and \v{S}ev\v{c}ovi\v{c} \cite{Pavlikova2022-WWW}). When the number of vertices increases we can observe less proportion of positively or negatively invertible graphs although all adjacency matrices are pseudo-invertible in the Moore-Penrose sense. It is the main purpose of this paper to introduce and analyze concept of positively and negatively pseudo-invertible graphs. We also present possible applications in the spectral graph theory.  Furthermore, we are concerned with spectral properties of real block symmetric matrices of the form 
\begin{equation}
\mathscr{M}  = \left(
\begin{array}{cc}
A & K\\
K^T & B
\end{array}
\right),
\label{matrixC}
\end{equation}
where $A$, and $B$ are $n\times n$, and $m\times m$ real symmetric matrices, and $K$ is an $n\times m$ real matrix. Such an adjacency matrix $\mathscr{M}$ corresponds to the graph $G_{\mathcal M}$ on $n+m$ vertices which is obtained by bridging the vertices of the graph $G_A$ to the vertices of $G_B$ through the $(n,m)$-bipartite graph $G_K$ (cf.  Pavl\'{\i}kov\'a and \v{S}ev\v{c}ovi\v{c} \cite{Pavlikova2019-NACO, Pavlikova2016-CMMS}), i.\,e.,  its adjacency matrix ${\mathcal M}$ of the graph $G_{\mathcal M}$ has the form of (\ref{matrixC}) (see Fig.~\ref{fig-bipartite}). Using the explicit form of the Moore-Penrose pseudo-inverse of the block matrix $\mathscr{M}$ we derive explicit form of pseudo-inverse graphs for graphs which are constructed from a given graph by adding pendant vertices or pendant paths. 

\begin{figure}[htp]
\begin{center}
\includegraphics[width=0.95\textwidth]{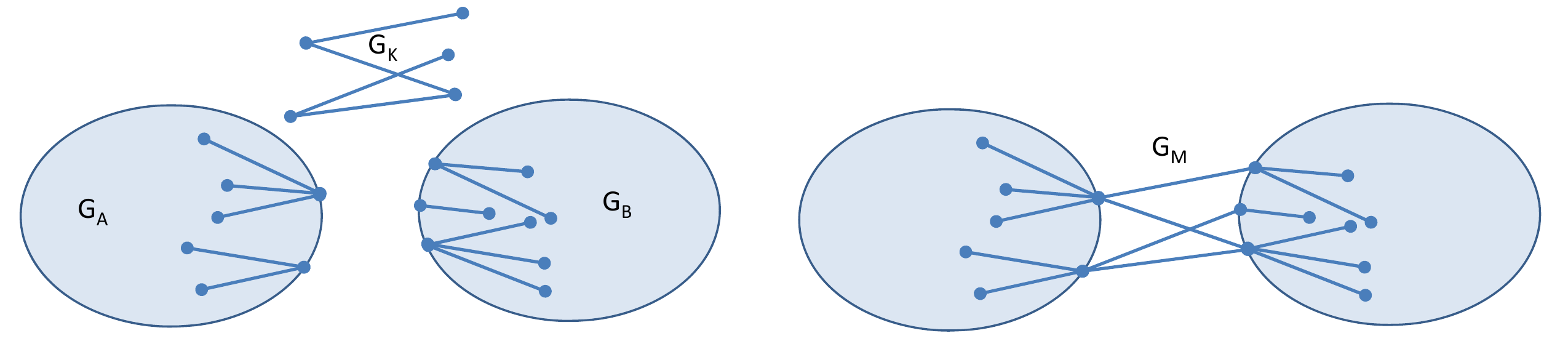}
\end{center}
\caption{
Graphs $G_A$, $G_B$, and a bipartite graph $G_K$ (left). A bridged graph through a bipartite graph $G_K$ (right).}
\label{fig-bipartite}
\end{figure}

The rest of the paper are organized as follows. In Section 2 we recall the classical concept of the Moore-Penrose pseudo-inversion $\mathscr{M}^\dagger$ of a matrix $\mathscr{M}$. Maximal and minimal eigenvalues of a pseudo-inverse matrix $\mathscr{M}^\dagger$ are closely related to the least positive $\lambda_+(\mathscr{M})>0$ and the largest negative eigenvalues $\lambda_-(\mathscr{M})<0$ of a symmetric real matrix $\mathscr{M}$. We present the key idea of derivation of the explicit Banachiewicz–Schur form of the Moore-Penrose pseudo-inverse of a block symmetric matrix. It can be used in order to characterize the least positive and the largest negative eigenvalues in terms of semidefinite optimization problem. Section 3 is devoted to pseudo-invertibility of simple connected graphs. We introduce a novel concept of positively and negatively pseudo-invertibility of a  graphs. We also present a summary table containing number of positively/negatively pseudo-invertible graphs. The list of graphs with no more than 10 vertices contains complete information on various classes of (pseudo)invertible graphs including their spectral properties. Furthermore, we present special examples of pseudo-invertible graphs constructed from a given graph by adding pendant vertices or pendant paths.

\section{Moore-Penrose pseudo-inversion of a block symmetric matrix and its application in spectral theory}

In this section we derive an explicit form of the Moore-Penrose pseudo-inverse matrix for a given block symmetric matrix satisfying certain assumptions on the entries of the block matrix. First, we recall the classical notion of the Moore-Penrose pseudo-inverse matrix. Let $\Lambda= diag\left(\lambda_{1}, \ldots, \lambda_{N}\right)$ be a diagonal matrix with real eigenvalues $\lambda_{i}\in \mathbb{R}, i,\ldots, N$. For a diagonal real matrix we can define the pseudo-inverse matrix as follows:
\[
\Lambda^{\dagger}=diag\left(\mu_{1}, \ldots, \mu_{N}\right) \quad where \quad \mu_{i}=1/\lambda_{i} \ \ if \  \lambda_{i} \neq 0, \  and, \  \mu_{i}=0 \  if \  \lambda_{i}=0.
\]
If $\mathscr{M}$ is an $N\times N$ real symmetric matrix, then it is orthogonally similar to a diagonal matrix $\Lambda$, i.\,e.,  there exists a matrix $\mathscr{P}$ such that $\mathscr{P}^{T}\mathscr{P}=\mathscr{P}\mathscr{P}^{T}=I$ and $\mathscr{P}\mathscr{M}\mathscr{P}^{T}=\Lambda=\operatorname{diag}\left(\lambda_{1}, \ldots, \lambda_{N}\right)$, $\lambda_{i}$ being real eigenvalues of $\mathscr{M}$. Let  $\mathscr{M}^{\dagger}$  be the Moore-Penrose pseudo-inverse matrix satisfying Moore-Penrose axioms for symmetric matrices, i.\,e., 
\begin{equation}
\mathscr{M}\mathscr{M}^{\dagger} = \mathscr{M}^{\dagger}\mathscr{M}, \quad \mathscr{M}\mathscr{M}^{\dagger}\mathscr{M}=\mathscr{M}, \quad \mathscr{M}^{\dagger}\mathscr{M}\mathscr{M}^{\dagger}=\mathscr{M}^{\dagger} .
\label{MP}
\end{equation}
Then $\mathscr{M}^{\dagger}$ can be uniquely defined as follows: 
$\mathscr{M}^{\dagger}
= \mathscr{P}^{T}\Lambda^{\dagger}\mathscr{P}$. Clearly, $\mathscr{M}^{\dagger}$ is a symmetric matrix, and $(\mathscr{M}^{\dagger})^\dagger=\mathscr{M}$. Furthermore, $\mathscr{M}^{\dagger}= \mathscr{M}^{-1}$ provided that the matrix $\mathscr{M}$ is invertible. If $K$ is an $n\times m$ real matrix, then its Moore-Penrose inverse matrix is an $m\times n$ matrix $K^\dagger$ which is uniquely determined by the Moore-Penrose identities:
\begin{equation}
(K K^{\dagger})^T = K K^{\dagger}, \quad (K^{\dagger} K)^T = K^{\dagger} K, \quad K K^{\dagger} K =K, \quad K^{\dagger} K K^{\dagger}= K^{\dagger} .
\label{MPmn}
\end{equation}
For further properties of generalized Moore-Penrose inverses we refer the reader to \cite{Ben}.

Suppose that $\mathscr{M}$ is a symmetric real matrix having positive and negative eigenvalues. This is a typical case when $\mathscr{M}$ represents an adjacency matrix of a simple graph with no loops having  zeros in its diagonal, and so the largest eigenvalue is positive whereas the least eigenvalue is negative. Clearly, the least positive $ \lambda_+(\mathscr{M})$ and largest negative $\lambda_-(\mathscr{M})$ eigenvalues of $\mathscr{M}$ can be expressed as follows (see e.g.  \cite{Pavlikova2016}):
\begin{equation}
  \lambda_+(\mathscr{M}) = \lambda_{max}(\mathscr{M}^\dagger)^{-1}, \qquad \lambda_-(\mathscr{M}) = \lambda_{min}(\mathscr{M}^\dagger)^{-1}.
\label{lambdaminmax}
\end{equation}
Here $\lambda_{max}(\mathscr{M}^\dagger)>0$ and $\lambda_{min}(\mathscr{M}^\dagger)= - \lambda_{max}(-\mathscr{M}^\dagger)<0$ are the maximal and minimal eigenvalues of the inverse matrix $\mathscr{M}^\dagger$, respectively. 

In what follows, we denote by $\R^{n\times m}$ the vector space consisting of all $n\times m$ real matrices. Suppose that $A\in \R^{n\times n}$, $B\in \R^{m\times m}$ are symmetric matrices, and $K\in \R^{n\times m}$. Consider the following $\left(n+m\right)\times \left(n+m\right)$ real symmetric
block matrix:
\begin{equation}
\mathscr{M}=\left(\begin{array}{cc}
A & K \\
K^{T} & B
\end{array}\right).
\label{blockM}
\end{equation}
The Moore-Penrose inverse $\mathscr{M}^\dagger$ of a block symmetric matrix $\mathscr{M}$ always exists and it has the block symmetric form:
\begin{equation}
\mathscr{M}^\dagger
=\left(\begin{array}{cc}
E & H \\
H^{T} & F
\end{array}\right).
\label{blockMinv}
\end{equation}
One can easily set up a system of matrix equations for the elements $E,F, H$ by following the basic axioms for the Moore-Penrose inverse. On the other hand, this system of nonlinear matrix equations need not have an explicit solution. Assuming further conditions on $A, B$, and $K$ there exists the so-called Banasiewicz-Schur explicit form of the Moore-Penrose pseudo-inverse matrix $\mathscr{M}^{\dagger}$ of $\mathscr{M}$. Although the form of  $\mathscr{M}^{\dagger}$ is known in the literature (see e.g. \cite{Castro}, \cite{Tian}), we give a short sketch of its derivation for reader's convenience. This way we recall the method of block diagonalization of  $\mathscr{M}$ and $\mathscr{M}^{\dagger}$ which is important matrix algebra tool for reformulation of the nonlinear eigenvalue problem as nonlinear optimization problem with semi-definite matrix inequality constraints. 
\color{black}
Generalized Moore-Penrose inverses of bordered matrices have been studied by F. Hall in a series of papers (see e.g. \cite{Hall1976}). In the context of the graph theory integral matrices play important role. Integral generalized inverses of integral matrices have been studied by Batigne, Hall, and Katz in \cite{Hall1978}. 
\color{black}

We apply the Schur complement trick. We define the matrix $\mathscr{Q}$ having its inverse matrix  $\mathscr{Q}^{-1}$ as follows:
\[
\mathscr{Q}=\left(\begin{array}{cc}
I & -K B^\dagger \\
0 & I
\end{array}\right), 
\qquad\qquad \mathscr{Q}^{-1}=\left(\begin{array}{cc}
I & K B^\dagger \\
0 & I
\end{array}\right) .
\]
Assuming the structural condition $K \left(I-B^\dagger B\right)=0$, a straightforward calculation yields the identity:
\begin{eqnarray*}
\mathscr{Q} \mathscr{M} \mathscr{Q}^{T} 
&=&\left(\begin{array}{cc}
S_{A} & K\left(I-B^\dagger B\right) \\
\left(I-BB^\dagger \right)K^{T} & B
\end{array}\right)  =
\left(\begin{array}{cc}
S_{A} & 0  \\
0 & B
\end{array}\right),
\end{eqnarray*}
where $S_{A}=A-K B^\dagger K^{T}$ is the generalized Schur complement of the matrix $B$ matrix in $\mathscr{M}$. Next, following the ideas of the Schur complement trick,  we propose the pseudo-inverse matrix in the so-called Banachiewicz–Schur form:
\begin{equation}
\mathscr{M}^{\dagger} = \mathscr{Q}^{T}
\left(\begin{array}{cc}
S_{A}^{\dagger} & 0  
\\
0 & B^\dagger
\end{array}\right)\mathscr{Q}
=
\left(\begin{array}{cc}
S_{A}^{\dagger} & -S_{A}^{\dagger}K B^\dagger   \\
-B^{\dagger}K^{T}S_{A}^{\dagger} & B^\dagger + B^{\dagger}K^{T}S_{A}^{\dagger}KB^{\dagger}
\end{array}\right).
\label{MA}
\end{equation}
Again, using the assumption  $K \left(I-B^\dagger B\right)=0$ we obtain 
\begin{eqnarray*}
\mathscr{M} \mathscr{M}^{\dagger} &=& 
\left(\begin{array}{cc}
S_A S_{A}^{\dagger} & (I- S_A S_{A}^{\dagger})K B^\dagger  \\
0  & B B^\dagger
\end{array}\right) .
\end{eqnarray*}
Hence, $\mathscr{M}\mathscr{M}^{\dagger} = \mathscr{M}^{\dagger}\mathscr{M}$ provided that the condition $(I- S_A S_{A}^{\dagger})K = 0$ is satisfied. Finally, one can  easily verify the remaining Moore-Penrose axioms, i.\,e.,  $\mathscr{M}\mathscr{M}^{\dagger}\mathscr{M}=\mathscr{M}$,  $\mathscr{M}^{\dagger}\mathscr{M}\mathscr{M}^{\dagger}=\mathscr{M}^{\dagger}$. 

\begin{definition}\label{KABcompatible}
Let $A\in \R^{n\times n}$ and $B\in \R^{m\times m}$ be symmetric matrices. We say that a matrix $K\in \R^{n\times m}$ is $(A,B)$ compatible, if
$K \left(I-B^\dagger B\right)=0$ and  $(I- S_A S_{A}^{\dagger})K = 0$ where $S_{A}=A-K B^\dagger K^{T}$.

\end{definition}

Analogously, exchanging the role of $A$ and $B$ matrices we conclude that the Moore-Penrose pseudo-inverse matrix  in the form:
\begin{eqnarray}
\mathscr{M}^{\dagger} &=& 
\left(\begin{array}{cc}
A^\dagger + A^{\dagger}KS_{B}^{\dagger}K^{T}A^{\dagger} & -A^{\dagger}KS_{B}^{\dagger} \\
-S_{B}^{\dagger}K^{T}A^{\dagger} & S_{B}^{\dagger}
\end{array}\right), 
\label{MB}
\end{eqnarray}
provided that $\left(I-A A^\dagger \right) K=0$, and $K (I-  S_{B}^{\dagger} S_B) = 0$. Since the Moore-Penrose pseudo-inverse matrix $\mathscr{M}^\dagger$ is unique, we obtain the following proposition.

\begin{proposition}\label{theo1}
Let $A\in \R^{n\times n}$ and $B\in \R^{m\times m}$ be symmetric matrices, $K\in \R^{n\times m}$. Assume that $K \left(I-B^\dagger B\right)=0$, $(I- S_A S_{A}^{\dagger})K = 0$, and $\left(I-A A^\dagger \right) K=0$, $K (I-  S_{B}^{\dagger} S_B) = 0$, i.\,e.,  $K$ is $(A,B)$ compatible and $K^T$ is $(B,A)$ compatible. Then for the generalized Schur complements $S_A= A -K B^\dagger K^T$, and $S_B=B-K^{T} A^\dagger K$ the following identities hold:
\begin{enumerate}
    \item \quad $S_{A}^{\dagger}KB^{\dagger} = A^\dagger KS_{B}^{\dagger}$,
    \item \quad 
    $S_{A}^{\dagger} = A^\dagger + A^{\dagger}KS_{B}^{\dagger}K^{T}A^{\dagger}$,
    \item \quad $S_{B}^{\dagger} = B^\dagger + B^{\dagger}K^{T}S_{A}^{\dagger}KB^{\dagger}$.
\end{enumerate}

\end{proposition}

\begin{remark}
The structural conditions $K \left(I-B^\dagger B\right)=0$, $(I- S_A S_{A}^{\dagger})K = 0$ are independent of each other. Indeed, let us consider the following simple example:
\[
A=A^\dagger
= \left(\begin{array}{cc}
1 & 0 \\
0 & 0
\end{array}\right),
\quad
B=B^\dagger
= \left(\begin{array}{cc}
0 & 0 \\
0 & 1
\end{array}\right),
\quad
K = \left(\begin{array}{cc}
0 & 1 \\
0 & 0
\end{array}\right).
\]
Then, $K (I-B^\dagger B)=0$, but $S_A=0$, and so $S_A^\dagger=0$,  $(I- S_A S_{A}^{\dagger})K = K \not= 0$. Similarly, $\left(I-A A^\dagger\right) K=0$ but $S_B=0$, and so $K (I- S_{B}^{\dagger} S_B)=K\not=0$.
\end{remark}

\begin{proposition}\label{prop-1}
Let $\mathscr{M}$ be a block symmetric matrix of the form (\ref{blockM}). If the matrix $\mathscr{M}$ is invertible and $K \left(I-B^\dagger B\right)=0$, then the generalized Schur complement matrix  $S_A= A- K B^\dagger K^T$ is invertible. 

\end{proposition}

\begin{proof}
Suppose that $S_A x =0$, i.\,e.,  $Ax- K B^\dagger K^T x =0$. Take $y=- B^\dagger K^T x$. Then $A x + K y=0$, and $K^T x  + B y = \left(I- B B^\dagger\right) K^T x =0$. Hence $\mathscr{M} (x,y)^T =0$. Since $\mathscr{M}$ is invertible, \color{black}we conclude that \color{black} $x=y=0$, and $S_A$ is invertible as well. 
\end{proof}

\begin{remark}
The converse statement is not true, in general. Indeed, let us consider $A=I, B=0$, and $K=0$. Then $\mathscr{M}$ is singular although $S_A=A=I$ is invertible. On the other hand, if both $A$ and $B$ are invertible matrices, then $\mathscr{M}$ is invertible if and only the Schur complement matrix  $S_A= A- K B^{-1} K^T$ is invertible (cf. \cite{Pavlikova2016}).
\end{remark}

In the following definition we introduce a novel concept of positive and negative pseudo-invertibility of a symmetric real matrix. 

\begin{definition}\label{def:pseudoinv-matrix}
Let $A$ be a symmetric real matrix. It is called positively pseudo-invertible if the Moore-Penrose pseudo-inverse matrix  $A^{\dagger}$ is signable to a nonnegative matrix. That is there exists a diagonal $\pm1$ signature  matrix $D$ such that the matrix $D A^{\dagger} D$  contains nonnegative elements only. It is called negatively pseudo-invertible if there exists a diagonal $\pm1$ signature  matrix $D$ such that the matrix $D A^{\dagger} D$  contains nonpositive elements only.
\end{definition}

Positive and negative (pseudo)invertibility of a symmetric real matrix generalizes the concept of a monotone matrix. Recall that a matrix $A$ is monotone in the sense of Collatz if and only if the component-wise inequality $Ax\ge 0$ implies $x\ge 0$. Every monotone matrix $A$ is invertible, and the inverse matrix $A^{-1} \ge 0$ contains nonnegative elements. As a consequence, for a monotone matrix $A$ the inequality $-b\le Ax \le b$ implies $-A^{-1} b \le x \le A^{-1} b$. 
\color{black}
As a consequence, $\Vert x\Vert_\infty \le \Vert A^{-1} b\Vert_\infty \le \Vert A^{-1}\Vert_1 \Vert b\Vert_\infty$ where $ \Vert B \Vert_1 =\sum_{i,j} |B_{ij}|$, and $\Vert b\Vert_\infty =\max_{i} |b_i|$.
\color{black}
In the following proposition we show that a similar property holds for positively or negatively pseudo-invertible matrices which need not be monotone in general. 

\begin{proposition}\label{monotone}
Suppose that $n\times n$ real symmetric matrix $A$ is positive (negative) pseudo-invertible. If $ -b \le A x \le b$ \color{black} for a nonnegative vector $b\ge0$,\color{black}then
\color{black}
\[
\Vert A^\dagger A x \Vert_\infty \le \Vert A^{\dagger}\Vert_1 \Vert b\Vert_\infty.
\]
\color{black}
In particular, if the matrix $A$ is invertible, then  
\color{black}
$\Vert x\Vert_\infty \le \Vert A^{-1}\Vert_1 \Vert b\Vert_\infty$.
\color{black}

\end{proposition}

\begin{proof}
Notice that vector component-wise inequality $ -b \le y \le b$ implies  $ -\color{black} b\color{black}  \le D y \le \color{black} b\color{black}$ for any diagonal signature matrix $D$ containing $\pm 1$ elements. Without loss of generality suppose that $A$ is positively pseudo-invertible, i.\,e.,  there exists a signature matrix $D$ such that $D A^\dagger D \ge 0$. Notice that for a diagonal matrix $D$ containing  $\pm1$ elements only, we have $D D = I$. It means that $D^{-1} = D$. The inequality $-b\le Ax \le b$ implies  $-\color{black}b\color{black}\le D A x \le \color{black}b\color{black}$, and so $-D A^\dagger D  b\le D A^\dagger D D A x \le D A^\dagger D  b$. As $D^2=I$, we obtain $-D A^\dagger \color{black} D\color{black} b\le D A^\dagger A x \le D A^\dagger \color{black} D\color{black} b$. Finally, the inequality 
 $-D A^\dagger \color{black} D\color{black} b\le  A^\dagger A x \le D A^\dagger \color{black} D\color{black} b$
implies
\color{black}
$\Vert A^\dagger A x\Vert_\infty \le \Vert D A^{\dagger}\Vert_1 \Vert D b\Vert_\infty = \Vert A^{\dagger}\Vert_1 \Vert b\Vert_\infty$, 
\color{black}
as claimed. 
\color{black}
In the case $A$ is negatively pseudo-invertible the proof of the estimate is similar. 
\color{black}
If $A$ is invertible, then $A^\dagger A= A^{-1} A =I$ and the proof of proposition follows. 
\end{proof}

\color{black}
\begin{proposition}\label{symmetric}
Suppose that an $N\times N$ real symmetric matrix $A$ is positive and negative pseudo-invertible. Then 

\begin{itemize}
\item[i)]  $D_+ A ^{\dagger} D_+ = - D_- A^{\dagger} D_-$ where $D_\pm$ are signature matrices signing $A^\dagger$ to nonnegative and nonpositive matrices, respectively.
\item[ii)] The spectrum $\sigma(A)$ is symmetric, i.e. $\lambda\in\sigma(A)$ iff $-\lambda\in\sigma(A)$.

\end{itemize}
\end{proposition}

\begin{proof}

Let $D_\pm$ be diagonal signature matrices such that $D_+ A ^{\dagger} D_+$ contains nonnegative elements and $D_- A^{\dagger} D_-$ contains nonpositive elements only. Since $(D_\pm A^{\dagger} D_\pm)_{ij} = (D_\pm)_{ii} (A^{\dagger})_{ij} (D_\pm)_{jj}$, we conclude that $(D_\pm A^{\dagger} D_\pm)_{ij}\not=0$ if and only if $(A^{\dagger})_{ij}\not=0$. As a consequence, we obtain $(D_+ A ^{\dagger} D_+)_{ij} = - (D_- A^{\dagger} D_-)_{ij}$ for each $i,j$, and the proof of the part i) follows. 

Suppose that $\lambda\in\sigma(A), \lambda\not=0$. Then $\mu=1/\lambda\in\sigma(A^\dagger)$, and there exists an eigenvector $x\not=0$ such that $A^\dagger x = \mu x$. As $D_+D_+= D_-D_- = I$, we obtain $A^\dagger = D_+ D_+ A^\dagger D_+ D_+ = -D_+ D_- A^\dagger D_- D_+ = -D A^\dagger D$ where $D=D_+ D_-=D_- D_+$ is a signature matrix (see the part i)). Then for the nontrivial vector $\tilde x= D x = D_+D_- x$ we obtain, $x=D\tilde x$, and
\[
A^\dagger \tilde x =  - D A^\dagger D \tilde x = - D A^\dagger x = - \mu D x = -\mu \tilde x.
\]
Hence $-\mu\in\sigma(A^\dagger)$. As a consequence, we obtain $-\lambda\in\sigma(A)$, as claimed in the part ii).
\end{proof}

\begin{remark}\label{remark-offdiagonal}
The converse statement to the part ii) is not true, in general. Indeed, let us consider the following $5\times 5$ block matrix  
\[
\mathscr{M} = \left(
\begin{array}{cc}
0 & K\\
K^T & 0
\end{array}
\right)
\quad \text{where} \quad K = \left(
\begin{array}{ccc}
1 & 1 & 0\\
0 & 1 & 1
\end{array}
\right). 
\]
The spectrum $\sigma(\mathscr{M}) = \{\pm\sqrt{3}, \pm 1, 0\}$ is symmetric. The matrix $\mathscr{M}$ represents an adjacency matrix of the bipartite path graph $\mathbb{P}_5$ shown in Fig.~\ref{fig-m56-bipartite-nonsignable} (left). It is a simple connected graph with smallest number of vertices with an adjacency matrix which is neither positively, nor negatively pseudo-invertible. 
\end{remark}

We say that an $N\times N$ matrix $\mathscr{M}$ is permutationally similar to an off-diagonal block matrix iff there exists a permutation matrix $P$ and an $n\times m$ real matrix $K$, $n+m=N$, such that 
\begin{equation}
P^T \mathscr{M} P = \left(
\begin{array}{cc}
0 & K\\
K^T & 0
\end{array}
\right).
\label{blockmatrix}
\end{equation}

\begin{theorem}\label{theo-simul}
Let $\mathscr{M}\in \R^{N\times N}$ be a real symmetric matrix.
\begin{itemize}
\item[i)] If $\mathscr{M}$ is a positive (negative) pseudo-invertible matrix which is permutationaly similar to a block off-diagonal matrix  (\ref{blockmatrix}), then $\mathscr{M}$ is also negative (positive) pseudo-invertible matrix. 

\item[ii)] If $\mathscr{M}$ is a positive and negative pseudo-invertible matrix, then $\mathscr{M}$ is permutationaly similar to a block off-diagonal matrix of the form (\ref{blockmatrix}). 

\item[iii)] An off-diagonal matrix $\mathscr{M}$ of the form (\ref{blockmatrix}) is positively and negatively pseudo-invertible if and only if  there are $n\times n$ and $m\times m$ signature diagonal matrices $D_A$ and $D_B$ such that the matrix $D_B K^\dagger D_A$ contains elements of the same sign. 

\end{itemize}

\end{theorem}

\color{black}

\begin{proof}
First, we note that for the Moore-Penrose inverse $\mathscr{M}^\dagger$ of the following  block symmetric matrix $\mathscr{M}$ we have the implication:
\begin{equation}
\mathscr{M} = \left(
\begin{array}{cc}
0 & K\\
K^T & 0
\end{array}
\right) \quad \Longrightarrow
\quad 
\mathscr{M}^\dagger = \left(
\begin{array}{cc}
0 & (K^\dagger)^T\\
K^\dagger & 0
\end{array}
\right),
\label{bipartite-matrix}
\end{equation}
where $K^\dagger$ is the uniquely determined Moore-Penrose inverse of the $n\times m$ matrix $K$. 

\color{black} In order to prove the part i),  \color{black}
let us assume there is an $N\times N$ permutation matrix $P$ such that $P^T \mathscr{M} P$ has the block form of (\ref{blockmatrix}). Without loss of generality, we may assume $P=I$ is the identical permutation of indices. Assume $\mathscr{M}$ is positively pseudo-invertible. We shall prove  that $\mathscr{M}$ is also negatively  pseudo-invertible, and vice versa. Assuming that $\mathscr{M}$ is positively pseudo-invertible, then there exists  a signature matrix $D_+ = diag(D_A, D_B)$ such that the matrix
\begin{eqnarray}
\label{DADB}
D_+ \mathscr{M}^\dagger D_+ &=&
\left(
\begin{array}{cc}
D_A & 0  \\
0 & D_B\\
\end{array}
\right)
\left(
\begin{array}{cc}
0 & (K^{\dagger})^T  \\
K^{\dagger} & 0\\
\end{array}
\right)
\left(
\begin{array}{cc}
D_A & 0  \\
0 & D_B\\
\end{array}
\right)
\\
&=&
\left(
\begin{array}{cc}
0 & D_A (K^{\dagger})^T D_B   \\
D_B K^{\dagger} D_A & 0\\
\end{array}
\right)\nonumber
\end{eqnarray}
contains nonnegative elements only. Here $D_A$ and $D_B$ are $n\times n$ and $m\times m$  signature diagonal matrices containing $\pm1$ elements only. Taking the signature diagonal matrix $D_- = diag(D_A, - D_B)$ the matrix 
\begin{equation}
D_- \mathscr{M}^{\dagger} D_- =
\left(
\begin{array}{cc}
0 & -D_A (K^{\dagger})^T D_B   \\
-D_B K^{\dagger} D_A & 0\\
\end{array}
\right)
\label{DBDA}
\end{equation}
contains nonpositive elements only. Hence the matrix $\mathscr{M}$ is also negatively pseudo-invertible, and vice versa.

\color{black} In order to prove the part ii),  \color{black}
suppose that $\mathscr{M}$ is positively and negatively pseudo-invertible matrix.  
\color{black}
With regard to Proposition~\ref{symmetric} we obtain 
\begin{equation}
 D_+ \mathscr{M}^{\dagger} D_+ = - D_- \mathscr{M}^{\dagger} D_- ,
 \label{Dpm}
\end{equation}
where $D_\pm$ are diagonal signature matrices. 
\color{black}
Then there exists an $N\times N$ permutation matrix $P$ such that
\[
 P^T D_+ D_- P =
\left(
\begin{array}{cc}
I_n & 0   \\
0 & -I_m\\
\end{array}
\right),\ \ i.\,e., 
\quad
D_+ D_- = D_- D_+ =
P \left(
\begin{array}{cc}
I_n & 0   \\
0 & -I_m\\
\end{array}
\right) P^T ,
\]
where $I_n$, $I_m$ are the $n\times n$ and $m\times m$ identity matrices, $n+m=N$. Here $n$ ($m$) is the number of positive (negative) units in the matrix $D_+ D_-$, respectively. It follows from (\ref{Dpm}) that
\[
\mathscr{M}^{\dagger} = - D_+ D_-  \mathscr{M}^{\dagger} D_- D_+ =
- P \left(
\begin{array}{cc}
I_n & 0   \\
0 & -I_m\\
\end{array}
\right)
P^T \mathscr{M}^{\dagger} P
\left(
\begin{array}{cc}
I_n & 0   \\
0 & -I_m\\
\end{array}
\right)
P^T.
\]
Since $P^T P = P P^T =I$, we have
\[
P^T \mathscr{M}^{\dagger} P
=
- \left(
\begin{array}{cc}
I_n & 0   \\
0 & -I_m\\
\end{array}
\right)
P^T \mathscr{M}^{\dagger} P
\left(
\begin{array}{cc}
I_n & 0   \\
0 & -I_m\\
\end{array}
\right)
\]
Writing $P^T \mathscr{M}^{\dagger} P$ in the form of a block matrix we obtain
\begin{eqnarray*}
 P^T \mathscr{M}^{\dagger} P \equiv
 \left(
\begin{array}{cc}
V & H   \\
H^T & W\\
\end{array}
\right)
&=&
- \left(
\begin{array}{cc}
I_n & 0   \\
0 & -I_m\\
\end{array}
\right)
 \left(
\begin{array}{cc}
V & H   \\
H^T & W\\
\end{array}
\right)
\left(
\begin{array}{cc}
I_n & 0   \\
0 & -I_m\\
\end{array}
\right)
\\
&=&
 \left(
\begin{array}{cc}
-V & H   \\
H^T & -W\\
\end{array}
\right).
\end{eqnarray*}
It means that $V=W=0$. Hence
\[
P^T \mathscr{M}^{\dagger} P =
\left(
\begin{array}{cc}
0 & H   \\
H^T & 0\\
\end{array}
\right)
\ \Longrightarrow \
P^T \mathscr{M} P =
\left(
\begin{array}{cc}
0 & K   \\
K^T & 0\\
\end{array}
\right).
\]
where $K=(H^T)^{\dagger}$. It means that the matrix $\mathscr{M}$ is permutationally similar to the block matrix (\ref{blockmatrix}), as claimed.

\color{black}
Finally, 
\color{black}
consider the diagonal $N\times N$ signature matrix $D=diag(D_B, D_A)$. Then 
\[
D \mathscr{M}^\dagger D
= \left(
\begin{array}{cc}
0 & D_A (K^\dagger)^T D_B\\
D_B K^\dagger D_A& 0
\end{array}
\right)
\]
contains elements of the same sign iff $D_B K^\dagger D_A$ contains elements of the same sign, and the proof of \color{black} the part iii) \color{black}  follows. 
\end{proof}

In the following theorem we present sufficient conditions for a general block symmetric matrix to be positively or negatively pseudo-invertible. 

\begin{theorem}\label{theo-si}
Let $A\in \R^{n\times n}$ and $B\in \R^{m\times m}$ be real symmetric matrices, $K\in \R^{n\times m}$. Let $S_A= A - K B^\dagger K^T$ be the Schur complement of $B$ in the block matrix $\mathscr{M}$ \color{black}in (\ref{blockM})\color{black}. Assume that $S_A$ and $B$ are both positively (negatively) pseudo-invertible  signable to a nonnegative (nonpositive) matrices by the signature matrices $D_A$ and $D_B$, respectively. Assume that the real matrix $K\in \R^{n\times m}$ is $(A,B)$ compatible and such that $D_A K D_B$ contains elements of the same signs. Then the block matrix $\mathscr{M}$ of the form (\ref{blockM}) is positively (negatively) pseudo-invertible. 
\end{theorem}

\begin{proof}
First,  we assume both $S_A$ and $B$ are positively pseudo-invertible and the matrix $D_A K D_B$ is nonnegative. We consider the signature matrix $D=diag(D_A,-D_B)$. If  $D_A K D_B$ is nonpositive, then  we take $D=diag(D_A,D_B)$. In what follows, we shall prove that $\mathscr{M}^\dagger$ is signable to a nonnegative real matrix by the signature matrix $D$. We have
\[
D \mathscr{M}^\dagger D =
\left( 
\begin{array}{cc}
D_A S_A^\dagger D_A  &  D_A S_A^\dagger K  B^\dagger D_B \\
D_B  B^\dagger K^T  S_A^\dagger D_A  & D_B B^\dagger D_B + D_B B^\dagger K^T S_A^\dagger K B^\dagger D_B
\end{array}
\right) .
\]
Since $D_A D_A =I_n, D_B D_B = I_m$, we have
\begin{eqnarray*}
D_A S_A^\dagger K  B^\dagger D_B &=& (D_A S_A^\dagger D_A) (D_A K  D_B) (D_B B^\dagger D_B) \ge 0,
\\
D_B B^\dagger K^T S_A^\dagger K B^\dagger D_B
&=& (D_B B^\dagger D_B) (D_B K^T D_A) (D_A S_A^\dagger D_A) (D_A K D_B) (D_B B^\dagger D_B) \ge 0,
\end{eqnarray*}
are matrices with nonnegative elements only. Hence the matrix $\mathscr{M}$ is positively pseudo-invertible, as claimed. In the case $S_A$ and $B$ are both negatively pseudo-invertible the proof of negatively pseudo-invertibility of $\mathscr{M}$ is analogous. 
\end{proof}

In the rest of this section we apply the previous results for characterization of the least positive and largest negative eigenvalues of a block matrix (\ref{blockmatrix}). As usual, we let denote by $\preceq$ the L\"owner partial ordering on symmetric matrices, i.\,e.,  $A\preceq B$ iff the matrix $B-A$ is a positive semidefinite matrix, that is $B-A\succeq 0$. Following \cite{bova}, \cite{Cvetkovic2004}, and using the L\"owner ordering, the maximal positive and minimal negative eigenvalues of $\mathscr{M}^\dagger$ can be expressed as follows:
\[
0< \lambda_{max}(\mathscr{M}^\dagger) = \min_{\mathscr{M}^\dagger\preceq t I} t, \qquad 0 > \lambda_{min}(\mathscr{M}^\dagger) = \max_{ s I \preceq \mathscr{M}^\dagger} s .
\]

\color{black}
Suppose that $\mathscr{M}$ is a symmetric real matrix such that $\lambda_{min}(\mathscr{M}) <0<\lambda_{max}(\mathscr{M})$. With regard to (\ref{lambdaminmax}) we have $\lambda_+(\mathscr{M}) = \lambda_{max}(\mathscr{M}^\dagger)^{-1}, \lambda_-(\mathscr{M}) = \lambda_{min}(\mathscr{M}^\dagger)^{-1}$.
\color{black}
Introducing the new variables $\mu=1/t, \eta=-1/s$, we deduce that the \color{black}least \color{black}  positive and largest negative eigenvalues of the matrix  $\mathscr{M}$ can be expressed as follows:
\begin{equation}
 \lambda_+(\mathscr{M}) = \max_{\mu \mathscr{M}^\dagger\preceq I} \mu, \qquad \lambda_-(\mathscr{M}) = - \max_{- \eta \mathscr{M}^\dagger\preceq I} \eta.
\label{lambdapm}
\end{equation}

Let us denote by $\Lambda^{HL}_{sg}(\mathscr{M})= \lambda_+(\mathscr{M}) -  \lambda_-(\mathscr{M})$ and $\Lambda^{HL}_{ind}(\mathscr{M})= \max(|\lambda_+(\mathscr{M})|, |\lambda_-(\mathscr{M}|)$  the HOMO-LUMO spectral gap $\Lambda^{HL}_{sg}$ and the index $\Lambda^{HL}_{ind}$ of a symmetric matrix $\mathscr{M}$ representing  the difference between the least positive and largest negative eigenvalue of the matrix $\mathscr{M}$. In the context of the spectral graph theory the eigenvalues of the adjacency matrix representing an organic molecule play an important role. The spectral gap $\Lambda^{HL}_{sg}(\mathscr{M})$ is also referred to as the HOMO-LUMO energy separation gap of the energy of the highest occupied molecular orbital (HOMO) and the lowest unoccupied molecular orbital orbital (LUMO). The bigger the spectral gap is the molecule is more stable. Generalizing the results by Pavl\'{\i}kov\'a and \v{S}ev\v{c}ovi\v{c} \cite{Pavlikova2019-NACO, Pavlikova2016-CMMS}, the HOMO-LUMO spectral gap and index of the matrix $\mathscr{M}$ can be expressed in terms of the following nonlinear programming problem:
\begin{eqnarray}
\Lambda^{HL}_{sg}(\mathscr{M}) &=& \max_{\mu,\eta\ge0}  \quad \mu+\eta
\\
\Lambda^{HL}_{ind}(\mathscr{M}) &=& \max_{\mu,\eta\ge0}  \quad \max(\mu, \eta)
\\
&& \text{s.t.}\quad \mu \mathscr{M}^\dagger \preceq I,  \ \ -\eta \mathscr{M}^\dagger\preceq I\nonumber.
\label{homolumo}
\end{eqnarray}
Assume $\mathscr{M}$ has a block symmetric matrix of the form \color{black}(\ref{blockM})\color{black}. Then, for any $\mu\ge 0$, we have $\mu  \mathscr{M}^\dagger \preceq I$ if and only if 

\[
\mu \left(\begin{array}{cc}
S_A^\dagger  & 0 \\
0 & B^\dagger
\end{array}\right) = \mu (\mathscr{Q}^{-1})^T \mathscr{M}^\dagger \mathscr{Q}^{-1} \preceq (\mathscr{Q}^{-1})^T \mathscr{Q}^{-1} \quad\hbox{where}\ \  \mathscr{Q}^{-1}=\left(\begin{array}{cc}
I & K B^\dagger \\
0 & I
\end{array}\right).
\]
Therefore, 
\begin{equation}
\mu  \mathscr{M}^\dagger \preceq I 
\quad\Leftrightarrow \quad
\left(
\begin{array}{cc}
I - \mu S_A^\dagger &  K B^\dagger \\
B^\dagger  K^T &  I - \mu B^\dagger + B^\dagger K^T K  B^\dagger
\end{array}
\right) \succeq 0.
\label{ineqmu}
\end{equation}
Similarly, for $\eta\ge 0$, we have 
\begin{equation}
-\eta \mathscr{M}^\dagger \preceq I \quad\Leftrightarrow \quad
\left(
\begin{array}{cc}
I + \eta S_A^{\dagger} &  K B^{\dagger} \\
B^{\dagger} K^T &  I + \eta B^{\dagger} + B^{\dagger} K^T K B^{\dagger}
\end{array}
\right) \succeq 0.
\label{ineqeta}
\end{equation}
\color{black}
The semi-definite constraints appearing in the right-hand sides of (\ref{ineqmu}), and (\ref{ineqeta}) can be viewed as linear matrix inequalities. Such inequalities can be utilized in order to optimize the least positive and largest negative eigenvalues $\lambda_+(\mathscr{M})$, and $\lambda_-(\mathscr{M})$ of a block matrix $\mathscr{M}$ with respect to matrix elements $A,B$, and $K$, respectively (cf. \cite{Pavlikova2016-CMMS}, \cite{Pavlikova2019-NACO}).
\color{black}

\section{Applications of block matrix pseudo-inversion in the graph theory}

Inverse graphs are of interest in estimating the least positive eigenvalue in families of graphs, a task for which there appears to be lack of suitable bounds. However, if the graphs are invertible, then one can apply one of the (many) known upper bounds on largest eigenvalues of the inverse graphs instead (cf. \cite{Pavlikova1990, Pavlikova2015}). Properties of the spectra of inverse graphs can also be used to estimate the difference between the minimal positive and maximal negative eigenvalue (the so-called HOMO-LUMO gap) for structural models of chemical molecules, as it was done e.g. for graphene structures in \cite{YeKlein}.

Let $G$ be an undirected simple graph, possibly with multiple edges, and with a symmetric binary adjacency matrix $A$. Conversely, if $A$ is a binary symmetric matrix, then we will use the symbol $G_A$ to denote the graph with the adjacency matrix $A$. The spectrum $\sigma(G)$ of $G$ consists of eigenvalues (i.\,e.,  including multiplicities) of $A_G$ (cf. \cite{Cvetkovic1978, Cvetkovic1988}). If the spectrum does not contain zero, then the adjacency matrix $A$ is invertible. We begin with a definition of an integrally invertible graph.

\begin{definition}\label{def-integerinv}
A graph $G=G_A$ is said to be integrally invertible if the inverse $A^{-1}$ of its adjacency matrix exists and is integral.
\end{definition}

It is well known (cf.~\cite{KirklandAkb2007}) that a graph $G_A$ is integrally invertible if and only if $det(A)=\pm1$. Next,  we recall the classical concept graph positive invertibility introduced by Godsil in \cite{Godsil1985} which was extended by Pavl\'\i kov\'a and \v{S}ev\v{c}ovi\v{c} to negatively invertible graphs \cite{Pavlikova2016}. Note that, in such a case the inverse matrix $A^{-1}$ need not represent a graph as it may contain negative entries. To overcome this difficulty in defining the inverse graph we restrict our attention to those graphs whose adjacency matrix is positively or negatively (pseudo)invertible. Next, we extend the concept of positive/negative invertibility to graphs whose adjacency matrices are not invertible.

\begin{definition}\label{def:inv}
A  graph $G_A$ is called positively (negatively) integrally invertible  if $det(A)=\pm 1$, and $A^{-1}$ is signable to a nonnegative (nonpositive) integral matrix. If $D$ is the corresponding signature matrix, then the positive (negative) inverse graph $H=G_A^{-1}$ is defined by the adjacency matrix $A_H=D A^{-1} D$ ($A_H= - DA^{-1} D$).
\end{definition}

The concept of positive integral invertibility coincides with the original notion of integral invertibility introduced by Godsil \cite{Godsil1985}. Definition~\ref{def:inv} extends Godsil's concept to a larger class of integrally invertible graphs with inverses of adjacency matrices signable to nonpositive matrices. 

\begin{definition}\label{def:pseudoinv}
A  graph $G_A$ is called positively (negatively) pseudo-invertible if the Moore-Penrose pseudo-inverse matrix  $A^{\dagger}$ is 
\color{black}
signable to a nonnegative (nonpositive) matrix. 
\color{black}
If $D$ is the corresponding signature matrix, then the positively (negatively) pseudo-inverse graph $H=G_A^{\dagger}$ is defined by the weighted nonnegative adjacency matrix $A_H=D A^{\dagger} D$ ($A_H= - DA^{\dagger} D$).
\end{definition}

Note that the positively (negatively) pseudo-inverse graph $H^\dagger$ is defined by the nonnegative weighted adjacency matrix $D A^\dagger D$ ($-D A^\dagger D$). The matrix $(D A^\dagger D)^\dagger$ is signable by the same signature matrix, and $D (D A^\dagger D)^\dagger D = D D (A^\dagger)^\dagger D D = A$. One can proceed similarly if $G$ is a negatively pseudo-invertible graph. As a consequence, we obtain
\[
(G^\dagger)^\dagger = G.
\]
Furthermore, it is easy to verify by a simple contradiction argument that the weighted pseudo-inverse graph $G^\dagger_A$ is connected provided that the original graph $G_A$ is connected.

In Table~\ref{tab-1} we show the number of all connected graphs with $m\le10$ vertices, number of invertible, and number of integrally/positively/negatively invertible graphs. We employed McKay's list of connected graphs available from  \url{http://users.cecs.anu.edu.au/~bdm/data/graphs.html}. The complete list of positively, negatively pseudo-invertible and pseudo-invertible graphs on $m\le 10$ including their spectrum and signature matrices can be found at: \url{http://www.iam.fmph.uniba.sk/institute/sevcovic/inversegraphs/} (see \cite{Pavlikova2022-WWW}).

\begin{table}
\small
\caption{\small Number of invertible ($det(A)\not=0$) and integrally invertible ($det(A)=\pm 1$) simple connected graphs $G_A$ with $m\le 10$ vertices. The number of positively but not negatively invertible graphs ($+$signable). The number of negatively  but not positively  pseudo-invertible graphs ($-$signable). The number of  positively and negatively invertible and pseudo-invertible graphs ($\pm$signable).
}
 \label{tab-1}
\scriptsize
\hglue-0.5truecm\begin{tabular}{r||r|r|r|r|r|r|r|r|r}
                 &$m=2$&$m=3$&$m=4$&$m=5$&$m=6$&$m=7$&$m=8$  & $m=9$  & $m=10$  \\
\hline
all conn. graphs       & 1   & 2   & 6   & 21  & 112 & 853 & 11117 & 261080 &11716571 \\
\hline
\multicolumn{10}{c}{Invertible graphs} \\
\hline
$det(A)\not=0$   & 1   & 1   & 3   &  8  & 52  & 342 &  5724 & 141063 & 7860195 \\
\hline
\multicolumn{10}{c}{Integrally invertible graphs} \\
\hline
int. invertible      & 1   & -   & 2   &  -  & 29  &   - &  2381 & -      & 1940904 \\
$+$signable          & 0   & -   & 1   &  -  & 20  &   - &  1601 & -      & 1073991 \\
$-$signable          & 0   & -   & 0   &  -  &  4  &   - &  235  & -      & 105363  \\
$\pm$signable        & 1   & -   & 1   &  -  &  4  &   - &  25   & -      & 349     \\
\hline
\multicolumn{10}{c}{Signable pseudo-invertible graphs} \\
\hline
$+$signable          & 0   & 0   & 1   &  3  & 27  & 111 &  2001 & 15310  & 1247128 \\
$-$signable          & 0   & 0   & 0   &  1  &  7  &  60 &   638 & 11643  & 376137  \\
$\pm$signable        & 1   & 1   & 3   &  4  & 13  &  25 &    93 &   270  & 1243     \\
\hline

\end{tabular}

\smallskip
Source: own computations \cite{Pavlikova2022-WWW}, \url{http://www.iam.fmph.uniba.sk/institute/sevcovic/inversegraphs/}
\end{table}

\begin{theorem}\label{theo-bipartite}
Let $G_\mathscr{M}$ be a simple connected graph.
\color{black}
\begin{itemize}
    \item [i)] If $G_\mathscr{M}$ is a bipartite graph which is positively (negatively) pseudo-invertible, then it is also negatively (positively) pseudo-invertible. 
    \item[ii)] If $G_\mathscr{M}$ is a positively and negatively pseudo-invertible graph, then it is a bipartite graph.
\end{itemize}
\color{black}
\end{theorem}

\begin{proof}
\color{black}
The part i) is a consequence of Theorem~\ref{theo-simul}, i).  Indeed, the adjacency matrix $\mathscr{M}$  of a bipartite graph $G_\mathscr{M}$ is permutationaly similar to a block off-diagonal matrix  (\ref{blockmatrix}). 
With regard to Theorem~\ref{theo-simul}, i), if $\mathscr{M}$ is positive (negative) pseudo-invertible, then it also negative (positive) pseudo-invertible matrix, as claimed in the part i). 

If the adjacency matrix $\mathscr{M}$ is positive and negative pseudo-invertible, then $\mathscr{M}$ is permutationaly similar to a block off-diagonal matrix of the form (\ref{blockmatrix}) (see Theorem~\ref{theo-simul}, ii)), and the proof of the theorem follows. 
\color{black}
\end{proof}

\begin{remark}\label{remark3}
\color{black}
In general,  bipartitness of a graph does not imply its positive and negative (pseudo)invertibility. 
\color{black}
A smallest example is the path $\mathbb{P}_5$ (shown in Fig.~\ref{fig-m56-bipartite-nonsignable} (left)) which is a bipartite graph with the adjacency matrix of the form (\ref{blockmatrix}). It is neither positively nor negatively pseudo-invertible, and its spectrum is symmetric $\sigma(\mathbb{P}_5)=\{0, \pm 1, \pm  \sqrt{3}\}$ \color{black}(see also Remark~\ref{remark-offdiagonal}). \color{black}  In Fig.~\ref{fig-m56-bipartite-nonsignable} (middle) we show a bipartite graph $G_A$ on 6 vertices with symmetric spectrum $\sigma(A)=\{0,0, \pm 1.1756, \pm 1.9021\}$. Again it is neither positively nor negatively pseudo-invertible. In Fig.~\ref{fig-m56-bipartite-nonsignable} (right) we show a bipartite graph $G_A$ on 8 vertices which is integrally invertible with a symmetric spectrum $\sigma(A)=\{ \pm 2.5231, \pm 1.4413, \pm 0.5669, \pm 0.4851 \}$. This is a smallest example of an integrally invertible  bipartite graph which is neither positively nor negatively invertible. 
\end{remark}

\begin{figure}
    \centering
    
    \includegraphics[width=.3\textwidth]{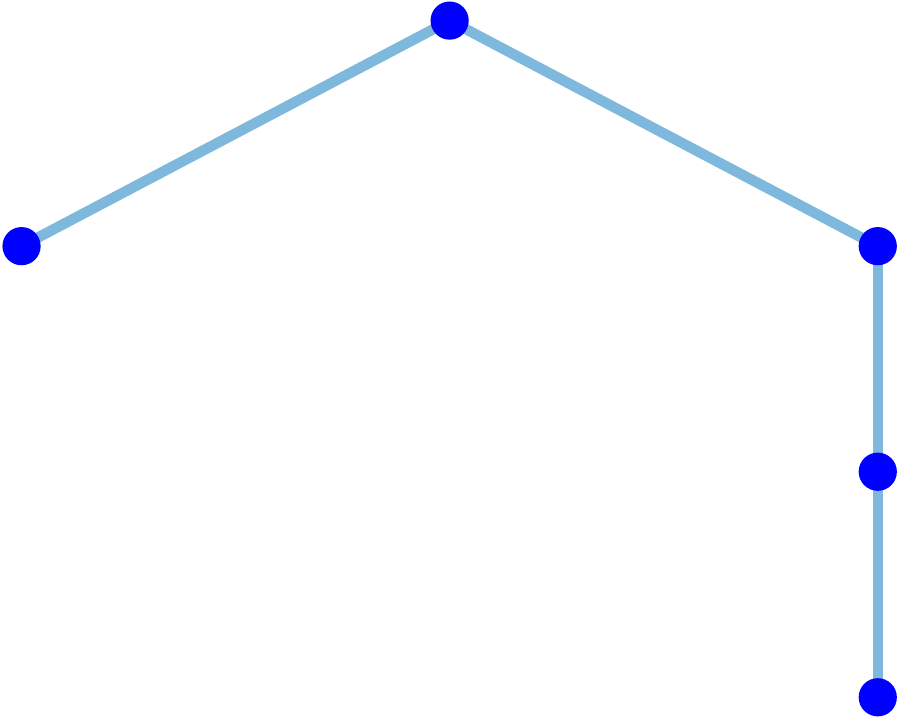}
    \quad
    \includegraphics[width=.3\textwidth]{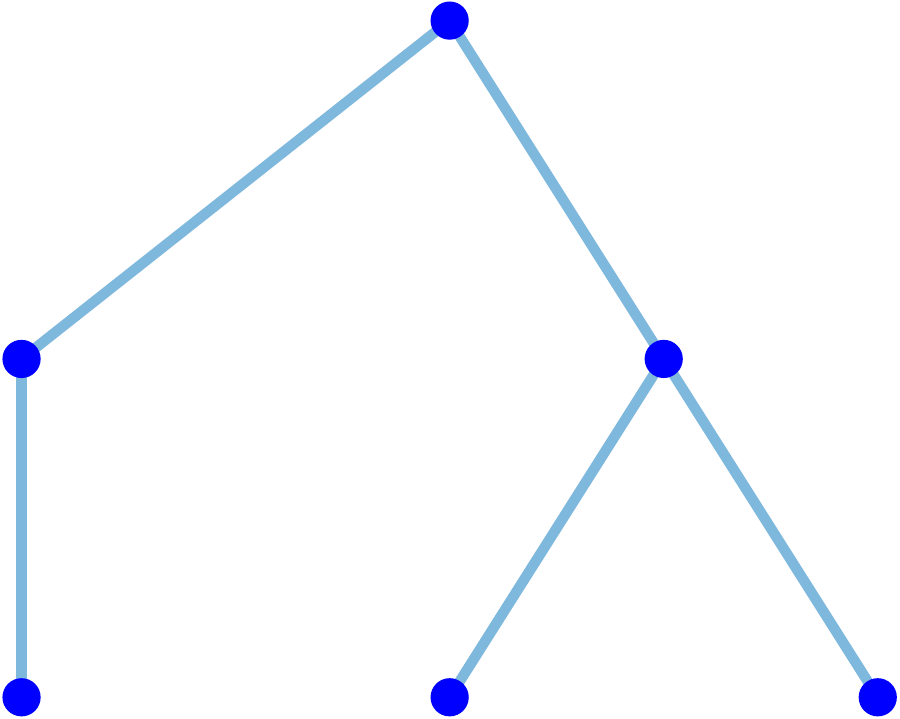}
    \quad
    \includegraphics[width=.3\textwidth]{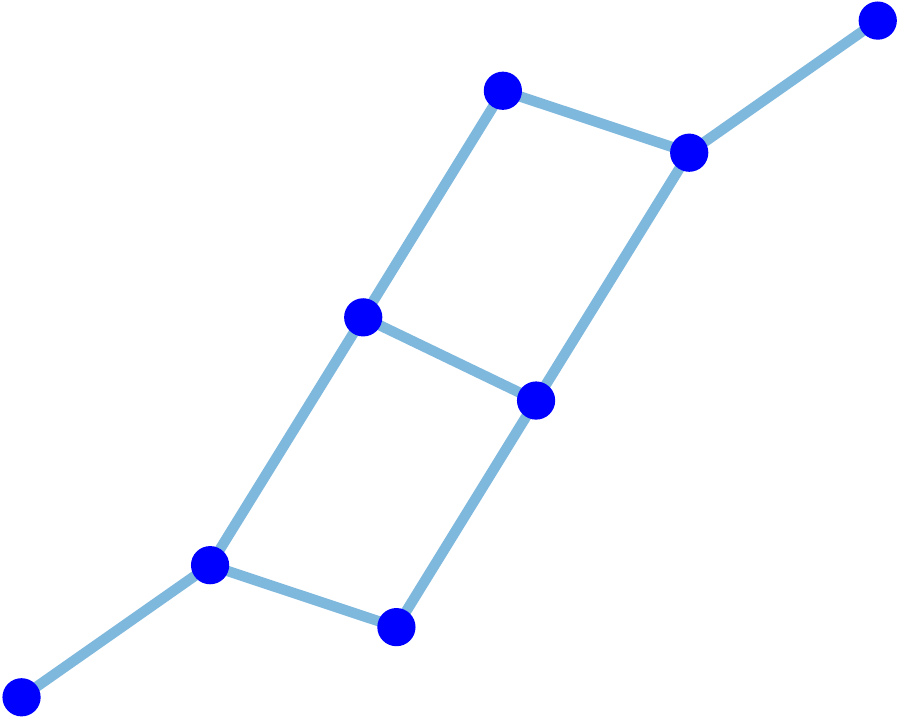}

    \caption{\small Examples of bipartite graphs which are neither positively nor negatively (pseudo)invertible.}
    \label{fig-m56-bipartite-nonsignable}
\end{figure}

\begin{example}\label{stargraph}
Let us denote $\mathbb{S}_{n+1}$ the bipartite star graph on $n+1$ vertices (see example in Fig.~\ref{fig-m5-pseudoinv} (right) for $n+1=6$). Its adjacency matrix $\mathscr{M}$ has the block form (\ref{blockmatrix}) with $K=(1, \ldots, 1)^T\in \mathbb{R}^n$. It is easy to verify that $K^\dagger =\frac{1}{n} K^T$. According to 
\color{black} Theorem~\ref{theo-simul}, iii), \color{black} $\mathbb{S}_{n+1}$ is positively and negatively pseudo-invertible. The adjacency matrix $\mathbb{S}_{n+1}^\dagger$ is equal to  $\frac{1}{n}\mathscr{M}$, and its spectrum $\sigma(\mathbb{S}_{n+1}) =\{ 0, \ldots, 0, \pm \sqrt{n}\}$.
\end{example}

\begin{remark}
According to the Theorem 8.8.2 due to Godsil and Royle \cite{Godsil2001} (see also Brouwer and Hamers \cite[Proposition 3.4.1]{Brouwer2012}) a graph is bipartite if and only if its spectrum is symmetric. In  Remark~\ref{remark3} we showed the the bipartitness (or symmetry of the spectrum) does not imply positively or negatively (pseudo)invertibility.
\end{remark}

\color{black}
If $A$ is an $m\times n$ matrix and $B$ is a $p\times q$ matrix, then the Kronecker product $A\otimes B$ is the $pm\times qn$ block matrix with blocks $(a_{ij} B)_{i=1,...,m,\  j=1,...,n}$. 
\color{black}

\begin{proposition}\label{borderedmatrix}
Assume $B\in \R^{m\times m}$ is a symmetric nonsingular matrix. Let us consider the block matrix
\begin{equation}
\mathscr{M} = \left(
\begin{array}{cc}
0 & K\\
K^T & B
\end{array}
\right),
\label{blockmatrixGk}
\end{equation}
where $K=(I, \ldots, I)^T \in \R^{km\times m}$, $I$ is the $m\times m$ identity matrix, and $0$ is the $n\times n$ zero matrix, $n=k m$ where $k\in\mathbb{N}$. Then the spectrum consists of real eigenvalues:
\[
\sigma(\mathscr{M}) = \{ \lambda\in \mathbb{R}, \ \lambda =0, \ \text{or}\ \ \lambda=(\mu\pm \sqrt{\mu^2 +4 k})/2, \ \text{for some}\ \mu\in\sigma(B) \}. 
\]
Furthermore, the spectrum $\sigma(\mathscr{M})$ is symmetric if and only if the spectrum $\sigma(B)$ is symmetric. Finally, 
\begin{equation}
\mathscr{M}^\dagger  = \frac{1}{k^2}
\left(
\begin{array}{cc}
- {\bf 1} \otimes B & k K\\
k K^T & 0
\end{array}
\right)
\label{blockmatrixGkinv}
\end{equation}
where $\otimes$ denotes the Kronecker product of matrices where ${\bf 1}$ is the $k\times k$ matrix consisting of ones. The matrix $\mathscr{M}$ is negatively pseudo-invertible.  If there exists an $m\times m$ signature matrix $D_-$ such that $D_- B D_- \le 0$, then $\mathscr{M}$ is also positively pseudo-invertible. 
\end{proposition}

\begin{proof}
It is easy to verify that $\lambda$ is a nonzero eigenvalue of $\mathscr{M}$ iff $\mu=\lambda - k/\lambda$ is an eigenvalue of $B$. Hence the spectrum $\sigma(\mathscr{M})$ is symmetric if and only if the spectrum $\sigma(B)$ is symmetric, and all nonzero eigenvalues of  $\mathscr{M}$ are given by $\lambda=(\mu\pm \sqrt{\mu^2 +4 k})/2$ for some  $\mu\in\sigma(B)$. For the Schur complement $S_A$ we have
\[
S_A = A - K B^{-1} K^T = - K B^{-1} K^T = - {\bf 1} \otimes B^{-1}, 
\]
where $\bf 1$ is the $k \times k$ matrix consisting of units. For example, if $k=2$, then  we have
\[
S_A = A - K B^{-1} K^T = 
- \left(
\begin{array}{cc}
B^{-1} & B^{-1}\\
B^{-1} & B^{-1}
\end{array}
\right) .
\]
Then it is easy to verify that 
\[
S_A^\dagger = - \frac{1}{k^2} {\bf 1} \otimes B \quad \text{and} \quad S_A S_A^\dagger = \frac{1}{k}  {\bf 1} \otimes I .
\]
Therefore, $(I- S_A S_A^\dagger) K = 0_{n\times m}$ and $ K(I-B^{-1} B) =0$, i.\,e.,  $K$ is a $(0,B)$ compatible matrix. Furthermore, $S_A^\dagger K B^{-1}= -\frac{1}{k} K$, and $B^{-1} + B^{-1} K^T S_A^\dagger K B^{-1} = B^{-1} -\frac{1}{k} B^{-1} K^T K = B^{-1} - B^{-1} =0$. Now, it follows from (\ref{MA}) that $\mathscr{M}^\dagger$ is given by (\ref{blockmatrixGkinv}). In particular, if $k=2$, then
\[
\mathscr{M}^\dagger  = \frac{1}{4}
\left(
\begin{array}{ccc}
- B  & - B & 2 I\\
- B  & - B & 2 I\\
2 I & 2 I  & 0
\end{array}
\right) .
\]
Hence $\mathscr{M}^\dagger$ is signable to nonpositive matrix by the signature matrix $D=diag(I, \ldots, I, -I)$. Finally, if there exists an $m\times m$ signature matrix $D_-$ such that $D_- B D_- \le 0$, then $\mathscr{M}$ is signable to nonnegative matrix by the signature matrix $D=diag(D_-, \ldots, D_-, D_-)$, as claimed.
\end{proof}

\begin{remark}
$G_B$ is a bipartite graph iff there exists a signature matrix $D_-$ such that $D_-B  D_- \le 0$. Indeed, if we set $(D_-)_{ii} =1$ for a vertex $i$ belongs to first bipartition, and $(D_-)_{jj} = -1$ for a vertex $j$ belongs to the second bipartition, then $D_-B D_- \le 0$. 
\color{black}
On the other hand, if there exists a signature matrix $D_-$ such that $D_-B  D_- \le 0$, then there exists a bipartition of $G_B$, each of two bipartions consisting of vertices with the same sign of $D_-$. 
\color{black}
\end{remark}

\begin{proposition}\label{graphpendantvertices}
Assume $G_B$ is a graph on $m$ vertices whose adjacency matrix $B$ is invertible. Let us denote by $G_B^{k,1}$ the graph which is constructed from $G_B$ by adding $k\in\mathbb{N}$ pendant vertices to each vertex of $G_B$. Then $G_B^{k,1}$ is negatively pseudo-invertible. If there exist an $m\times m$ signature matrix $D_-$ such that $D_- B D_- \le 0$, then $G_B^{k,1}$ is also positively pseudo-invertible graph.
\end{proposition}

\begin{proof}
The adjacency matrix $\mathscr{M}$ of $G_B^{k,1}$ has the form (\ref{blockmatrixGkinv}). The rest of the proof follows from Proposition~\ref{borderedmatrix}. 
\end{proof}

\color{black}
The following result is a consequence of Proposition~\ref{graphpendantvertices} for the case of $k=1$ pendant vertex added to vertices of $G_B$. It is a generalization of \cite[Theorem 3.38]{Bapat2014} stating that, if a tree is invertible, then it is isomorphic to itself iff it is a corona tree obtained from a tree $G_B$ by adding a pendant vertex to every vertex of $G_B$.
\color{black}

\begin{corollary}
Assume $G_B$ is a graph on $m$ vertices with an adjacency matrix $B$. Then the graph $G^{1,1}_B$ is negatively self pseudo-invertible, i.\,e.,  $(G^{1,1}_B)^\dagger \cong G^{1,1}_B$. \color{black} If $G_B$ is a bipartite graph, then $G^{1,1}_B$ is also positively self pseudo-invertible.
\color{black}

\end{corollary}
\medskip 

In Fig.~\ref{fig-P22-pseudoinv} and Fig.~\ref{fig-F022-pseudoinv} we show the simple path graph $G_B=\mathbb{P}_2$ on two vertices and the graph graph $G_B=F_0$ representing fulvene organic molecule. We also show the graphs $G_B^{2,1}$ constructed from $G_B$ by adding two pendant vertices and their negatively pseudo-invertible graphs $(G_B^{2,1})^\dagger$. For $G_B=\mathbb{P}_2$ the graph $(G_B^{2,1})^\dagger$ is also positively pseudo-invertible bipartite graph.

\begin{figure}
    \centering
    \includegraphics[width=.25\textwidth]{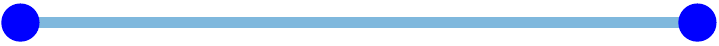}
    \quad
    \includegraphics[width=.3\textwidth]{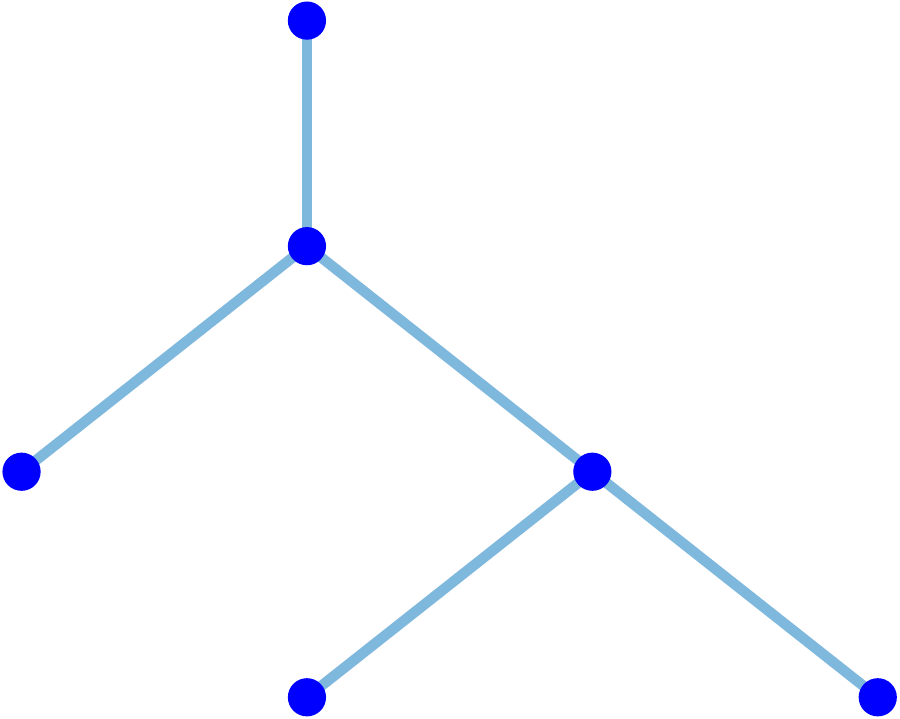}
    \quad
    \includegraphics[width=.3\textwidth]{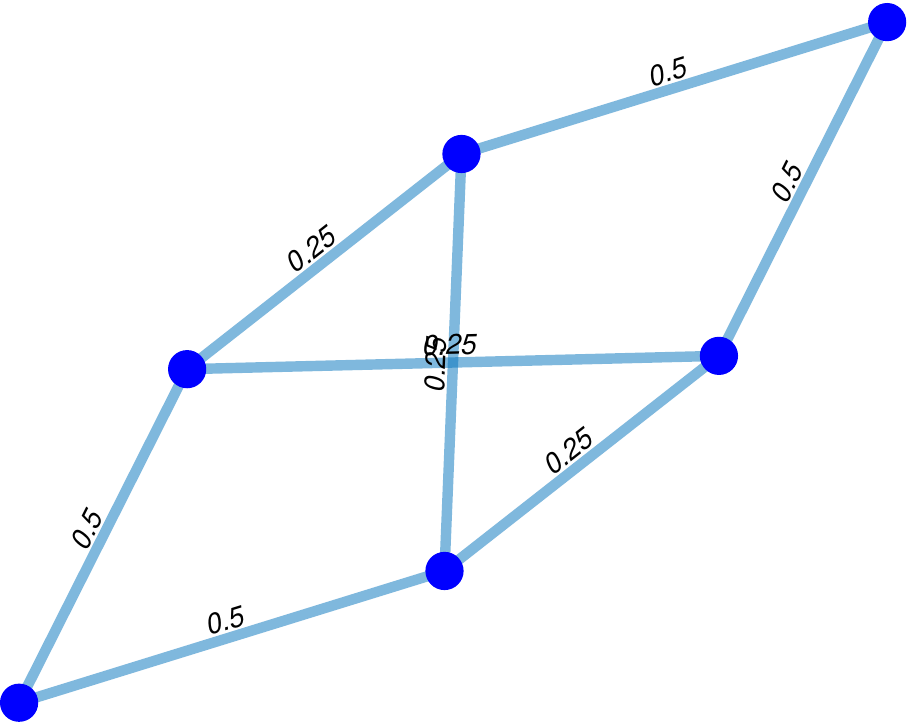}

    \caption{\small The path graph $G_B=\mathbb{P}_2$ with two vertices (left), the graph $G_B^{2,1}$ constructed from the path graph $\mathbb{P}_2$ by adding two pendant vertices (middle), its positive pseudo-inverse graph $(G_B^{2,1})^\dagger$ (right).}
    \label{fig-P22-pseudoinv}
\end{figure}

\begin{figure}
    \centering
    \includegraphics[width=.3\textwidth]{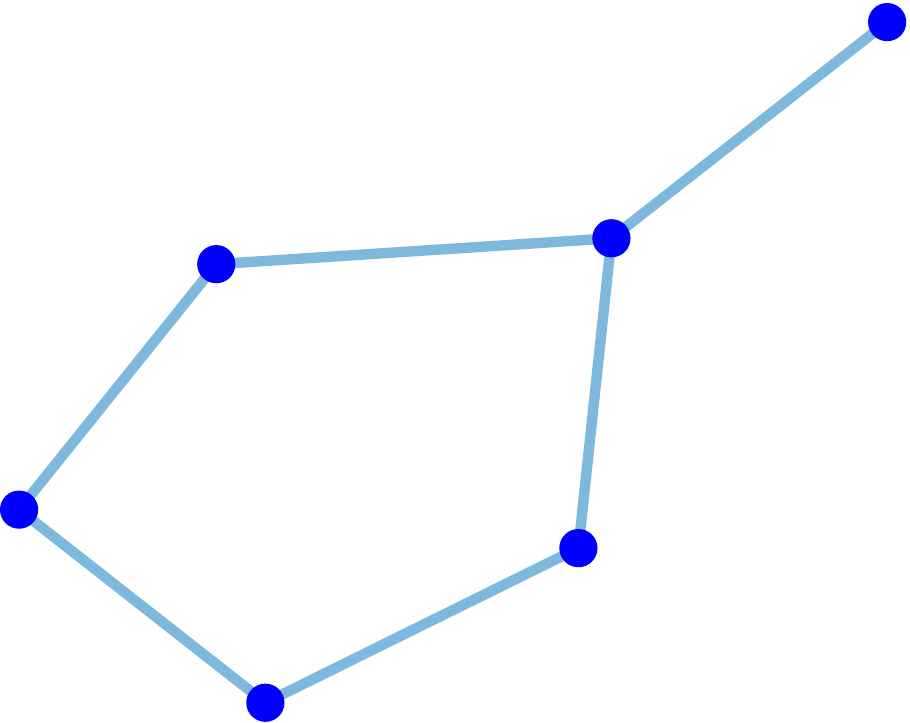}
    \quad
    \includegraphics[width=.3\textwidth]{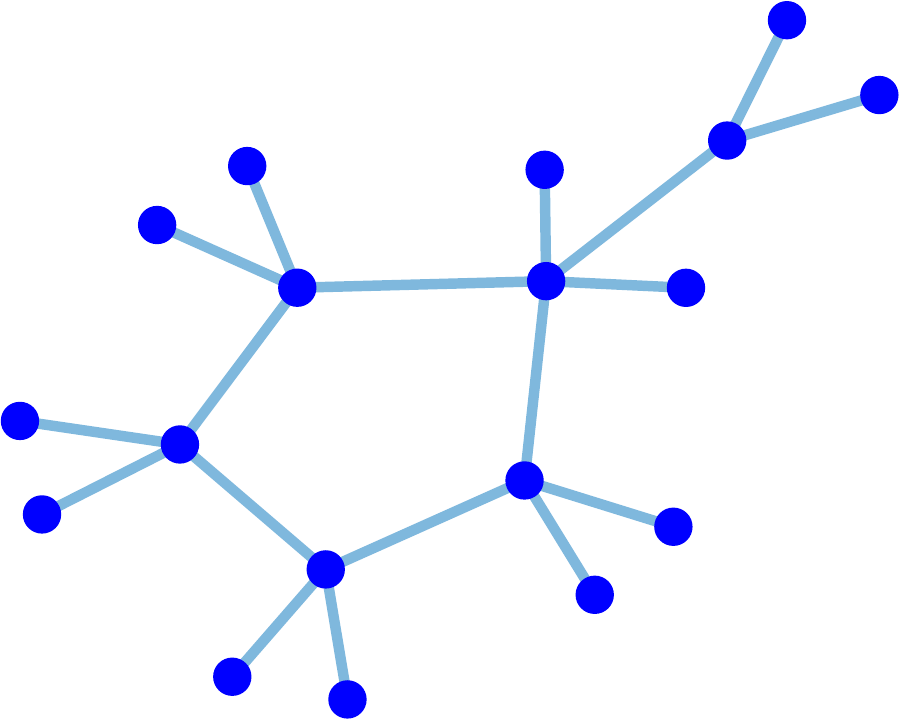}
    \quad
    \includegraphics[width=.3\textwidth]{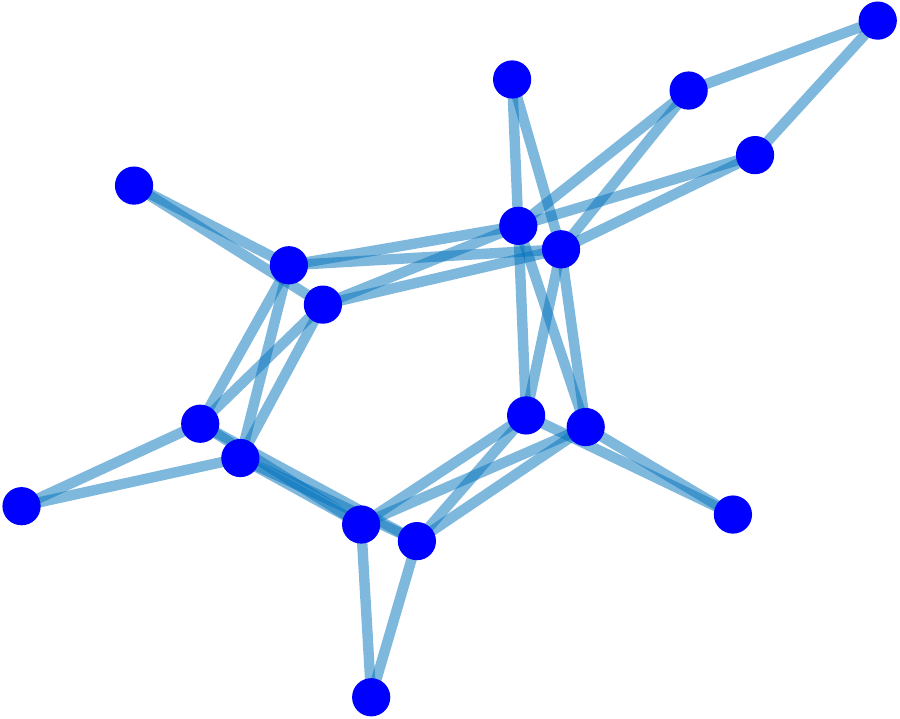}

    \caption{\small The negatively invertible fulvene graph $G_B=F_0$ (left), the graph $G_B^{2,1}$ constructed from $F_0$ by adding two pendant vertices (middle), its negative pseudo-inverse graph $(G_B^{2,1})^\dagger$ (right).}
    \label{fig-F022-pseudoinv}
\end{figure}

\begin{proposition}\label{graphpendantpathes}
Assume $G_B$ is a graph on $m$ vertices whose adjacency matrix $B$ is invertible.
Let us denote by $G_B^{1,l}, l\in\mathbb{N}$, the graph which is constructed from $G_B$ by adding pendant paths $\mathbb{P}_l$ to each vertex of $G_B$. Then the adjacency matrix $\mathscr{M}$ of $G_B^{1,l}$ is invertible. 
\begin{enumerate}
    \item If $l$ is odd, then $G_B^{1,l}$ is a negatively integrally invertible graph. Moreover, if there exists a signature matrix $D_-$ such that $D_- B D_- \le 0$, then $G_B^{1,l}$ is a positively integrally invertible graph.
    \item If $l$ is even, then $G_B^{1,l}$ is a positively/negatively invertible graph provided that $G_B$ is positively/negatively invertible graph.
\end{enumerate}

\end{proposition}

\begin{figure}
    \centering
    \includegraphics[width=.3\textwidth]{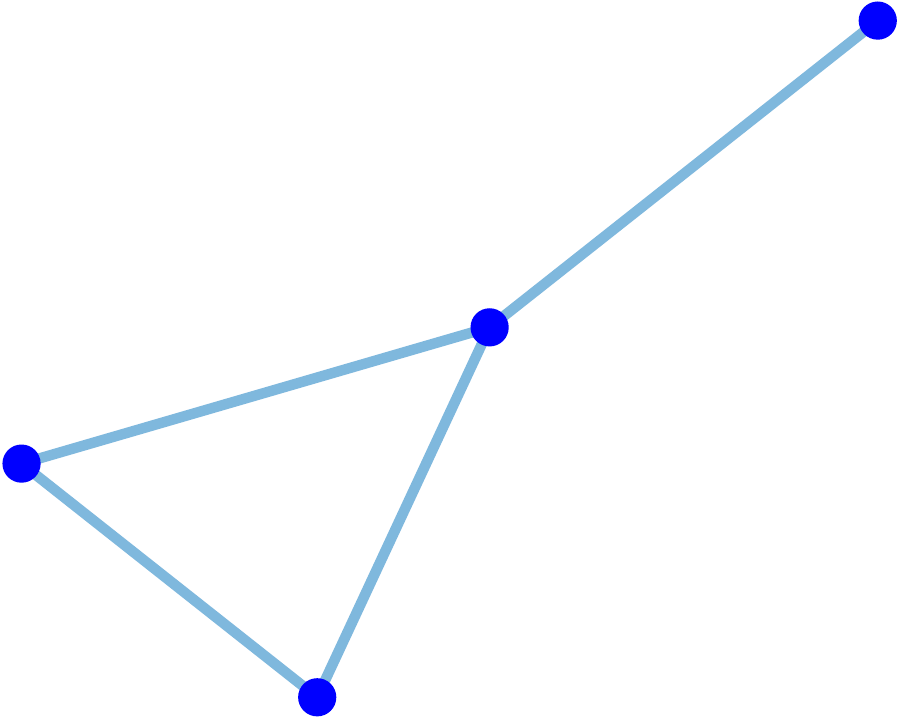}
    \quad
    \includegraphics[width=.3\textwidth]{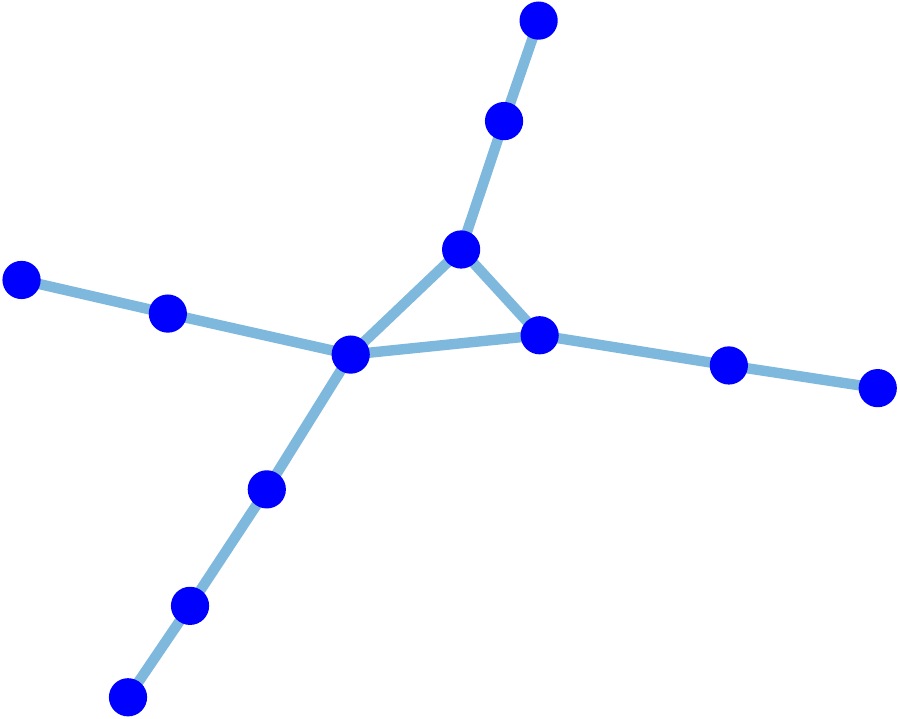}
    \quad
    \includegraphics[width=.3\textwidth]{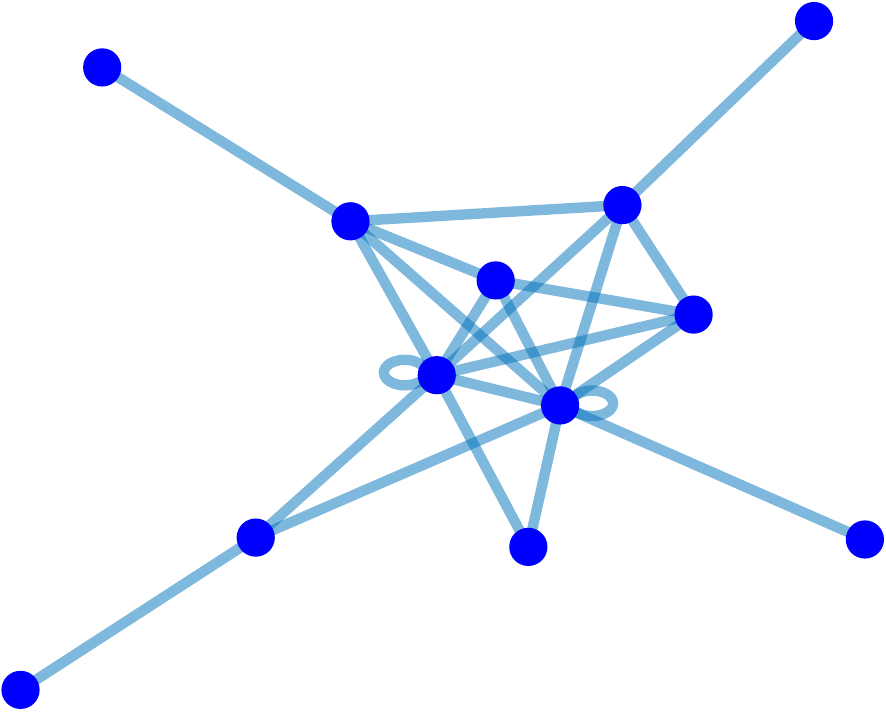}

    \caption{\small The positively invertible graph $G_B$ on four vertices (left), the graph $G_B^{1,2}$ constructed from $G_B$ by adding pendant paths $\mathbb{P}_2$ on each vertex (middle), its pseudo-inverse graph $(G_B^{1,2})^\dagger$ (right).}
    \label{fig-C12-pseudoinv}
\end{figure}

\begin{proof}
The adjacency matrix $\mathscr{M}\in \R^{(l+1)m\times (l+1)m}$ of the graph $G^{1,l}_B$ has the block form (\ref{blockM}) where the block matrices $A\in \R^{l m\times l m}$, and $K\in \R^{l m \times m}$ have the form:
\[
A = \mathscr{A} \otimes I = \begin{pmatrix}
    0 & I &   & & &  \\
    I & 0 & I & & &  \\
      & I & 0 & \ddots & &  \\
      &  &\ddots  &\ddots &  I & \\
      &  &  &  I & 0 &\\
    \end{pmatrix}
  \qquad K= e \otimes I = (I, 0, \dots, 0)^T,
\]
where $I=I_m$ is the identity matrix, $e=(1, 0, \dots, 0)^T\in \mathbb{R}^l$, and $\mathscr{A}$ is the adjacency matrix of the path graph $\mathbb{P}_l$, i.\,e.,  $\mathscr{A}_{ij}=1$ for $|i-j|=1$, and  $\mathscr{A}_{ij}=0$, otherwise. An example of a graph $G_B^{1,2}$ is shown in Fig.~\ref{fig-C12-pseudoinv}). The Schur complement matrix $S_A=A- K B^{-1} K^T = \mathscr{A} \otimes I  - e \otimes I B^{-1} e^T \otimes I = \mathscr{A} \otimes I  - e e^T \otimes B^{-1}$. It is straightforward to verify that the matrix $S_A$ is invertible, and it has the block cyclic form:
\[
\underbrace{S_A^{-1}}_{l=2}
= \begin{pmatrix}
0 & I \\
I & B^{-1}
\end{pmatrix},
\qquad
\underbrace{S_A^{-1}}_{l=3}
= \begin{pmatrix}
-B & 0 & B \\
0 & 0 & I \\
B & I & -B
\end{pmatrix},
\]
\[
\underbrace{S_A^{-1}}_{l=4}
= \begin{pmatrix}
0 & I & 0 & -I \\
I & B^{-1} & 0 & -B^{-1} \\
0 & 0 & 0 & I \\
-I & -B^{-1} & I & B^{-1}
\end{pmatrix},
\quad 
\underbrace{S_A^{-1}}_{l=5}=\begin{pmatrix}
-B & 0 & B & 0 & -B \\
0 & 0 & I & 0 & -I \\
B & I & -B & 0 & B \\
0 & 0 & 0 & 0 & I \\
-B & -I & B & I & -B
\end{pmatrix},
\]
etc. In general, if $l$ is odd, then the matrix $S_A^{-1}$ contains the matrix $B$, whereas, if $l$ is even, then it contains the inverse matrix $B^{-1}$. As the inverse matrices $B^{-1}$ and $S_A^{-1}$ exist then the matrix $K$ is $(A,B)$ compatible. Moreover, if $l$ is odd, then we have $S_A^{-1} K B^{-1} = (-I, 0, I, 0, \dots )^T$, and so $B^{-1} K^T S_A^{-1} K B^{-1} = 0$. If $l$ is even, then we have $S_A^{-1} K B^{-1} = (0, B^{-1}, 0, - B^{-1}, \dots )^T$, and so $B^{-1} K^T S_A^{-1} K B^{-1} = - B^{-1}$. The matrix $\mathscr{M}$ is invertible, and it has the block matrix form:
\[
\underbrace{\mathscr{M}^{-1}}_{l=2}
= \begin{pmatrix}
0 & I     & B^{-1}\\
I & B^{-1}& B^{-1}\\
0 & B^{-1}& B^{-1}\\
\end{pmatrix},
\qquad
\underbrace{\mathscr{M}^{-1}}_{l=3}
= \begin{pmatrix}
-B & 0 & B & I\\
0 & 0 & I  & 0\\
B & I & -B &-I\\
I & 0 & -I & 0
\end{pmatrix},
\]
\[
\underbrace{\mathscr{M}^{-1}}_{l=4}
= \begin{pmatrix}
0 & I & 0 & -I & 0 \\
I & B^{-1} & 0 & -B^{-1} & -B^{-1} \\
0 & 0 & 0 & I & 0\\
-I & -B^{-1} & I & B^{-1} &  B^{-1}\\
0 & -B^{-1} & 0 & B^{-1} & B^{-1} \\

\end{pmatrix},
\quad 
\underbrace{\mathscr{M}^{-1}}_{l=5}=\begin{pmatrix}
-B & 0 & B & 0 & -B & I \\
0 & 0 & I & 0 & -I  & 0\\
B & I & -B & 0 & B  & -I \\
0 & 0 & 0 & 0 & I   & 0\\
-B & -I & B & I & -B& I\\
I & 0 & -I & 0 & I  &0 
\end{pmatrix},
\]
etc. Now, if $l$ is odd, then the matrix $\mathscr{M}^{-1}$ is integral and $det(\mathscr{M})=\pm 1$. Clearly, the matrix $\mathscr{M}^{-1}$ is signable to a nonpositive matrix by the signature matrix: \\ 
$D=diag([I, [I, -I],-[I, -I], \dots, (-1)^{k-1} [I, -I], -I] )$ where $l=2k +1$. If there exists a signature matrix $D_-$ such that $D_- B D_- \le 0$, then the matrix $\mathscr{M}^{-1}$ is signable to a nonnegative matrix by the signature matrix: \\
$D=diag([D_-, -[ D_-, D_-], [ D_-, D_-], \dots, (-1)^{k} [D_-, D_-], D_-] )$ where $l=2k +1$.

Suppose that $l=2k$ is even. If $B$ is positively pseudo-invertible, then there exists a signature matrix $D_+$ su that $D_+ B^{-1} D_+ \ge 0$. Then the matrix $\mathscr{M}$ is positively pseudo-invertible because $D \mathscr{M}^{-1} D \ge 0$ where $D=diag([D_+, D_+], -[D_+, D_+], \dots, , (-1)^{k-1} [D_+, D_+], -D_+])$. If $B$ is negatively pseudo-invertible, then there exists a signature matrix $D_-$ su that $D_- B^{-1} D_- \le 0$. Then the matrix $\mathscr{M}$ is negatively pseudo-invertible as $D \mathscr{M}^{-1} D \le 0$ where $D=diag([D_-, -D_-], -[D_-, -D_-], \dots, , (-1)^{k-1} [D_-, -D_-], D_-])$, as claimed.

\end{proof}

\medskip


\begin{figure}
    \centering
    \includegraphics[width=.3\textwidth]{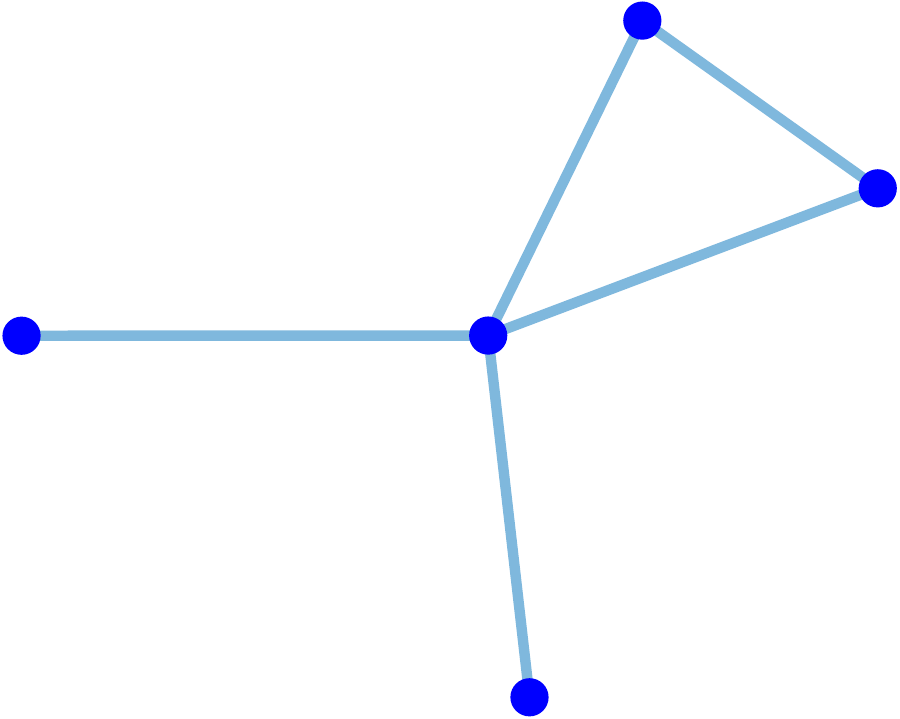}
    \quad
    \includegraphics[width=.3\textwidth]{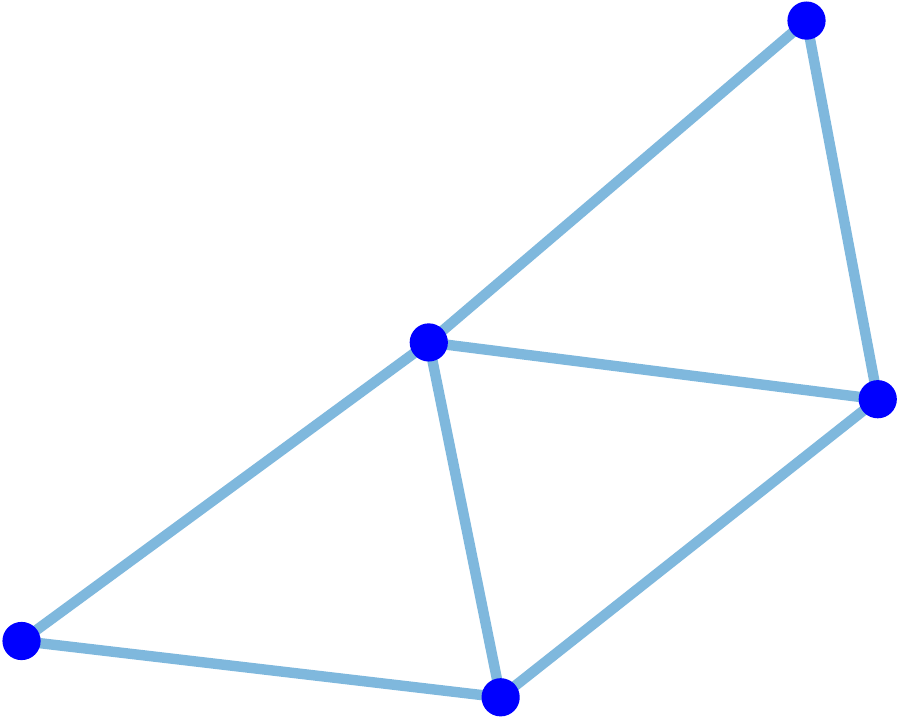}
    \quad
    \includegraphics[width=.3\textwidth]{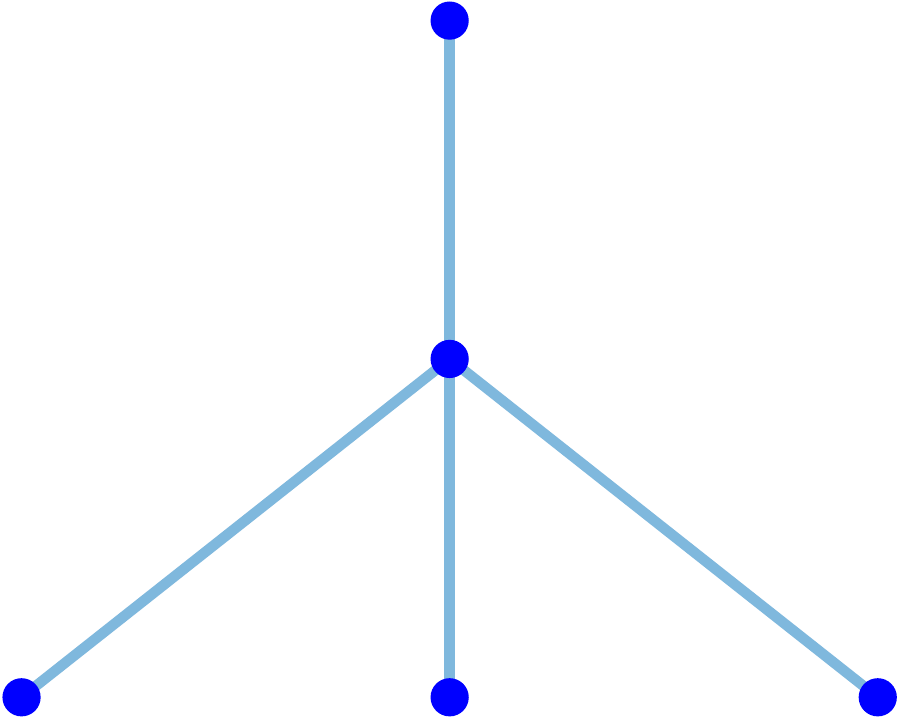}

\vglue 0.5truecm

    \includegraphics[width=.3\textwidth]{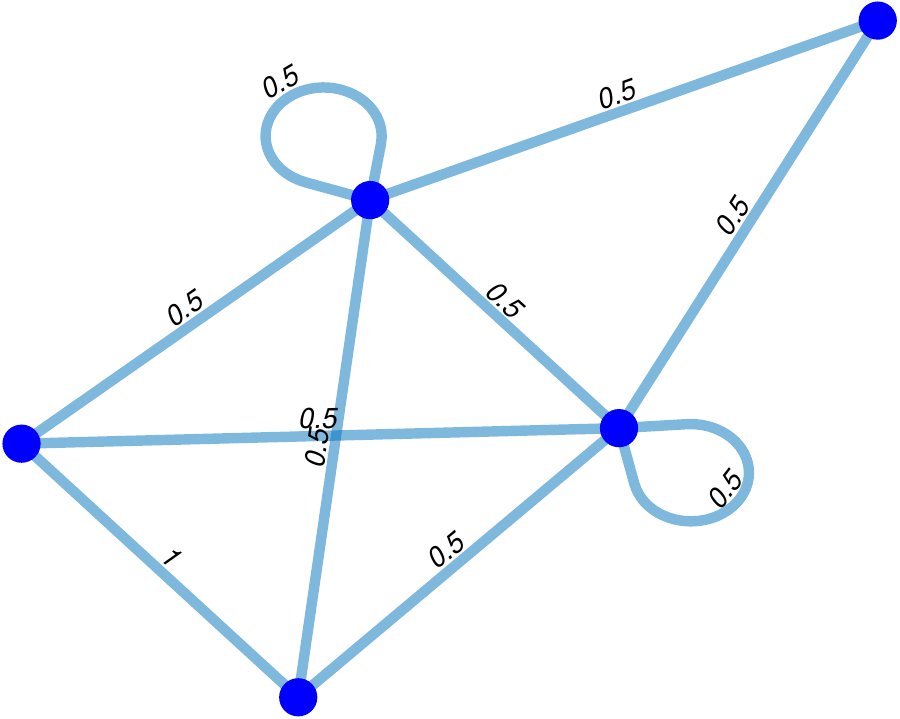}
    \quad
    \includegraphics[width=.3\textwidth]{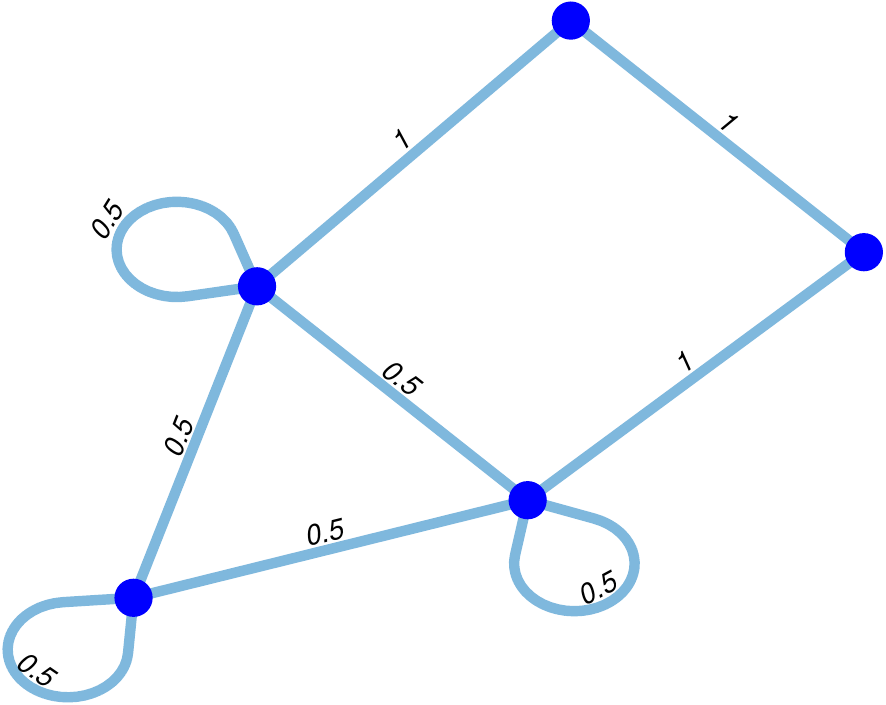}
    \quad
    \includegraphics[width=.3\textwidth]{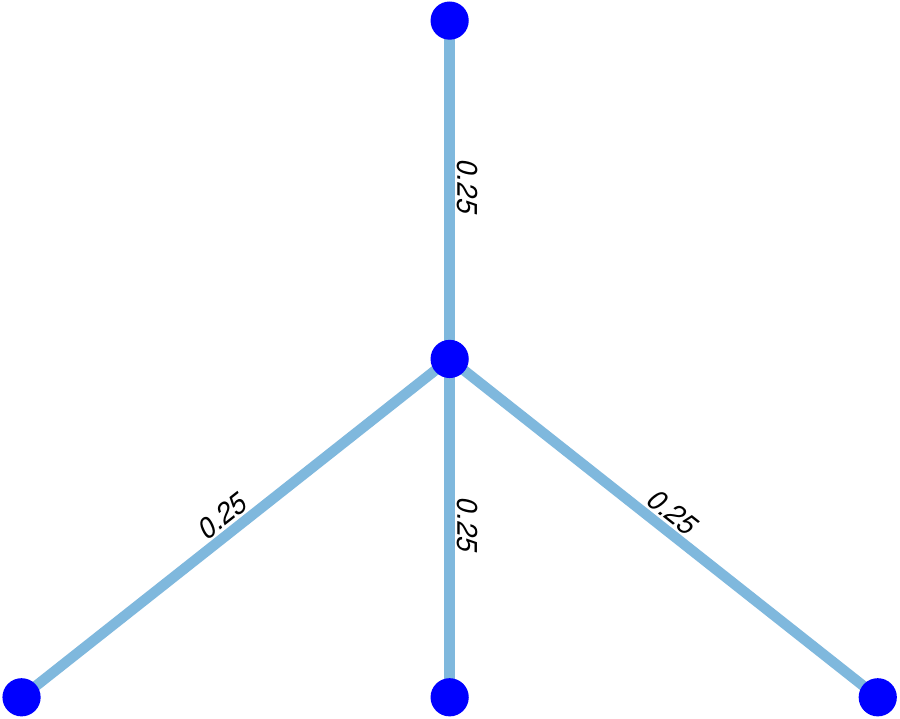}

    \caption{\small Top row: examples of positively (left), negatively (middle), positively and negatively (bipartite) pseudo-invertible graphs with  adjacency matrix $det(B)\not= \pm1$ on $m=5$  vertices. Corresponding graph pseudo-inversions (bottom row).}
    \label{fig-m5-pseudoinv}
\end{figure}

\begin{figure}
    \centering
    
    \includegraphics[width=.3\textwidth]{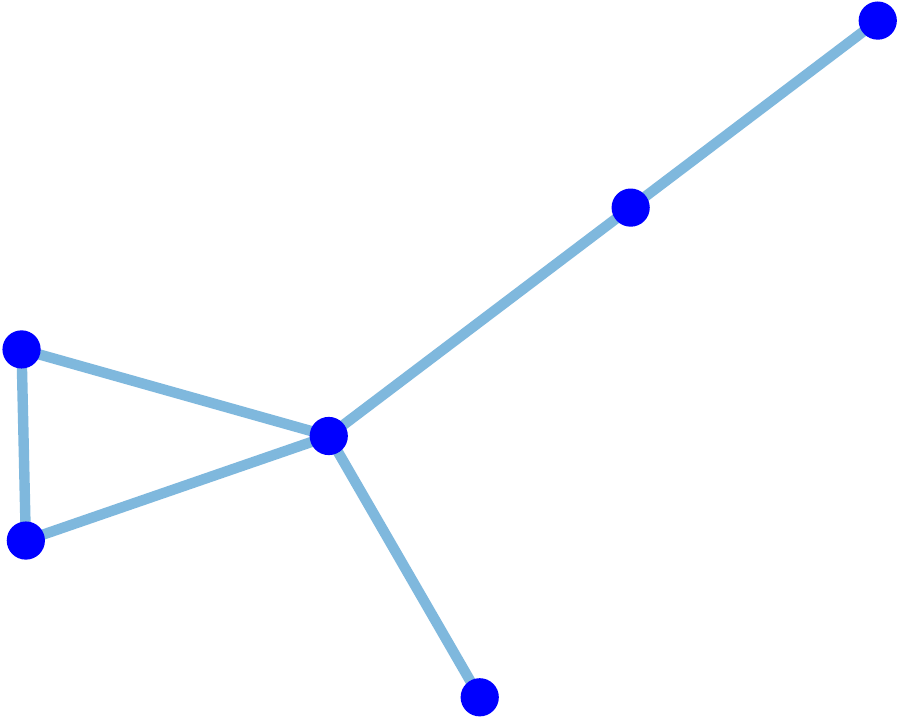}
    \quad
    \includegraphics[width=.3\textwidth]{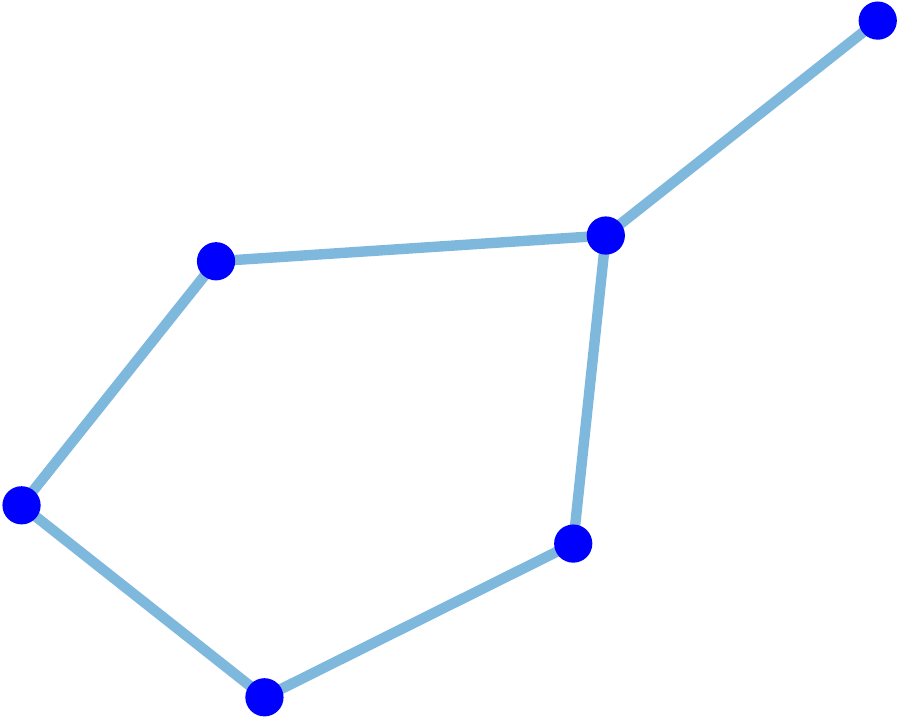}
    \quad
    \includegraphics[width=.3\textwidth]{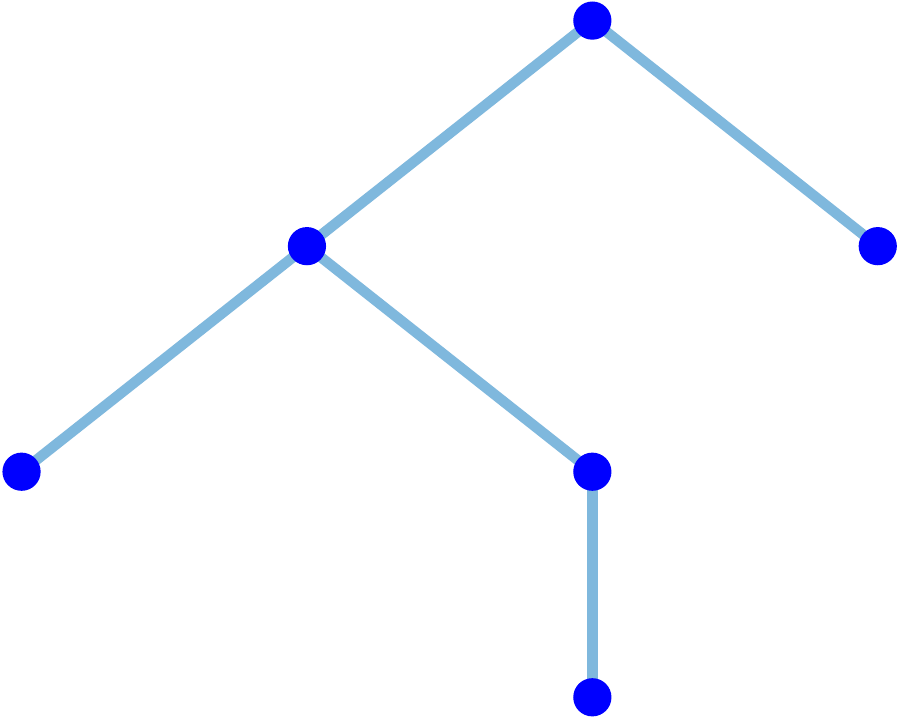}
    
\vglue 0.5truecm

    \includegraphics[width=.3\textwidth]{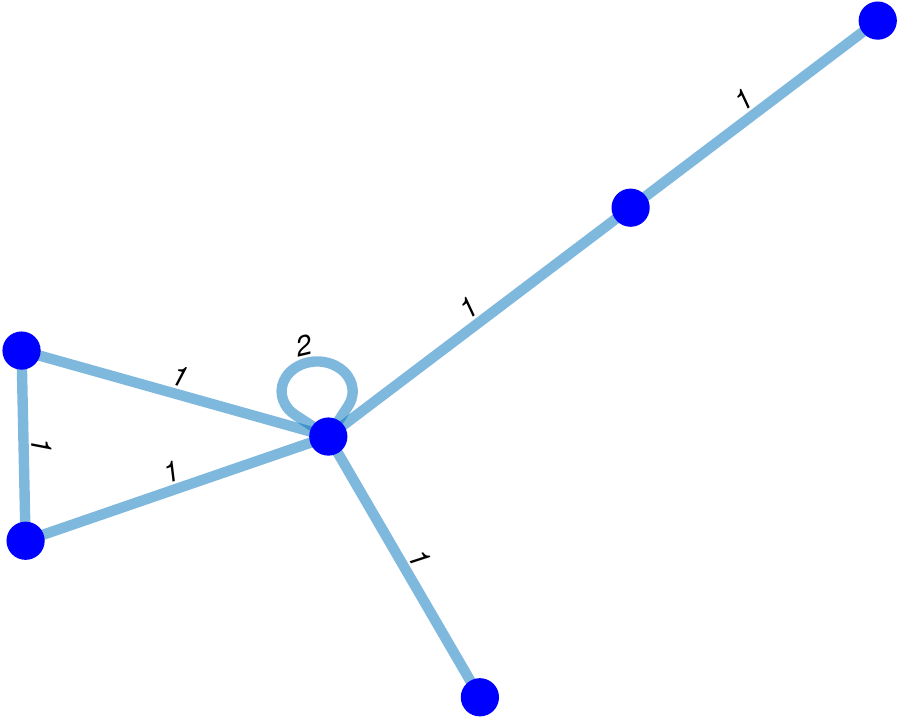}
    \quad
    \includegraphics[width=.3\textwidth]{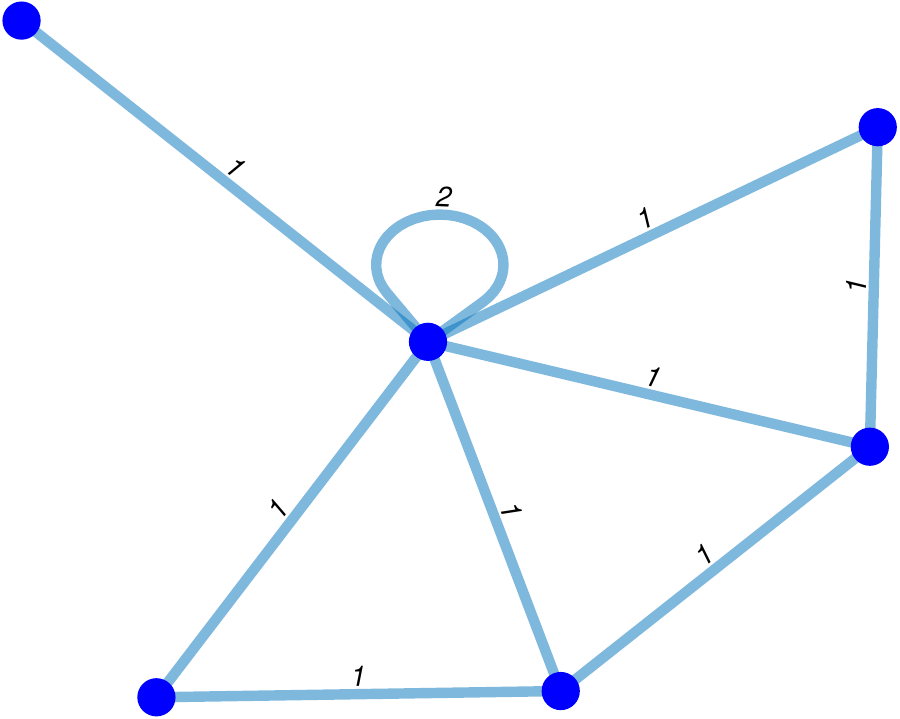}
    \quad
    \includegraphics[width=.3\textwidth]{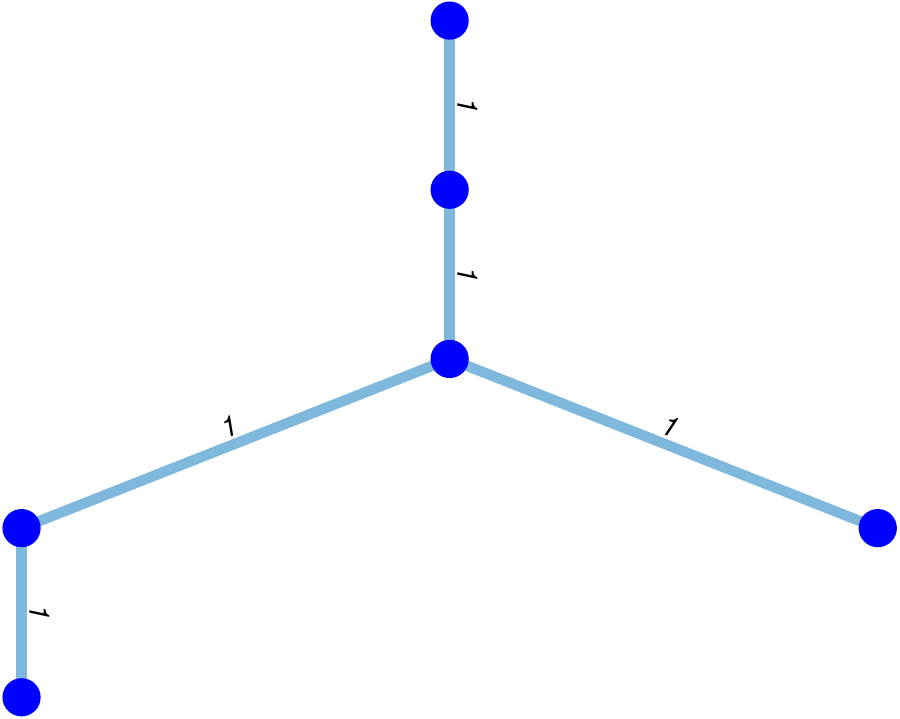}
    
    \caption{\small Top row: examples of positively (left), negatively (middle), positively and negatively (bipartite) integrally invertible graphs on $m=6$  vertices. Corresponding graph inversions (bottom row).}
    \label{fig-m6-intinv}
\end{figure}

In Fig.~\ref{fig-m6-intinv} (middle) we show the graph $F_0$ with 6 vertices representing the organic molecule of the fulvene hydrocarbon (5-methylidenecyclopenta-1,3-diene). The graph $F_0$ is negatively (but not positively) invertible with the integral inverse graph $(F_0)^{-1}$ depicted in Fig.~\ref{fig-m6-intinv} (middle). The nonsymmetric spectrum consists of the following eigenvalues $\sigma(F_0)= \{ -1.8608, -q, -0.2541, 1/q, 1, 2.1149 \}$ where $q=(\sqrt{5}+1)/2$ is the golden ratio with the least positive eigenvalue $\lambda_+(F_0)=1/q$.  The inverse adjacency matrix $A_{F_0}^{-1}$ is signable to a nonpositively integral matrix by the signature matrix $D = diag(-1,-1,1,1,1-1)$. The inverse graph is depicted in Fig.~\ref{fig-m6-intinv} (middle).

\begin{figure}
    \centering
    \includegraphics[width=.3\textwidth]{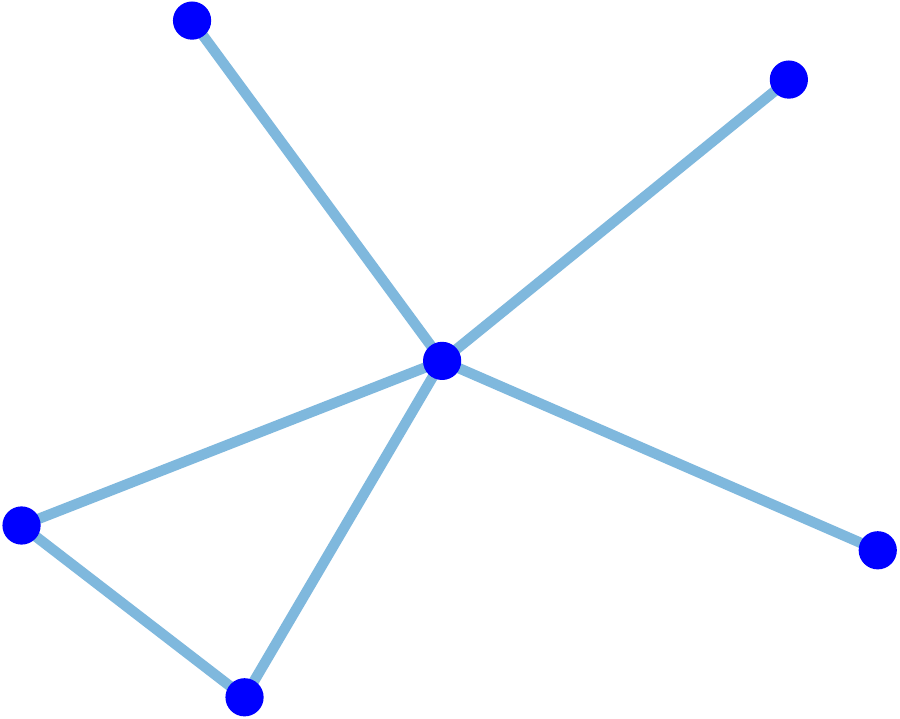}
    \quad
    \includegraphics[width=.3\textwidth]{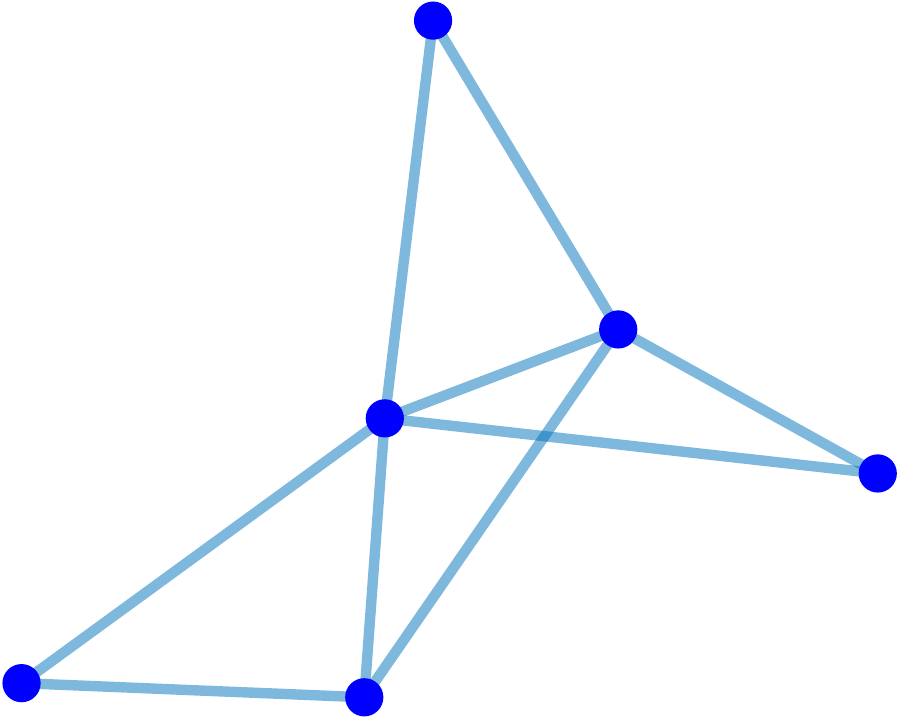}
    \quad
    \includegraphics[width=.3\textwidth]{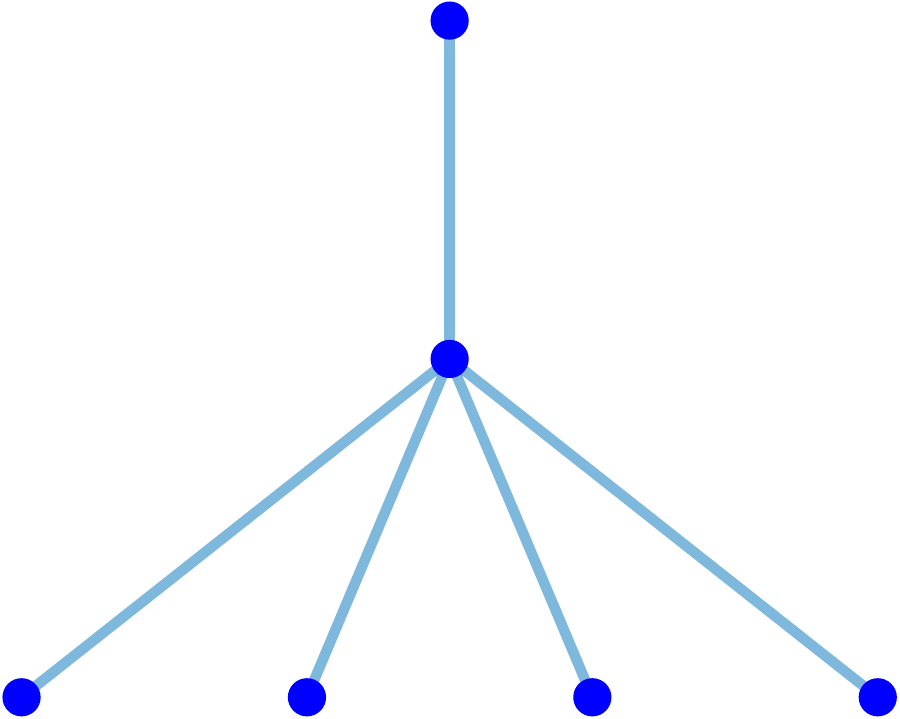}

\vglue 0.5truecm

    \includegraphics[width=.3\textwidth]{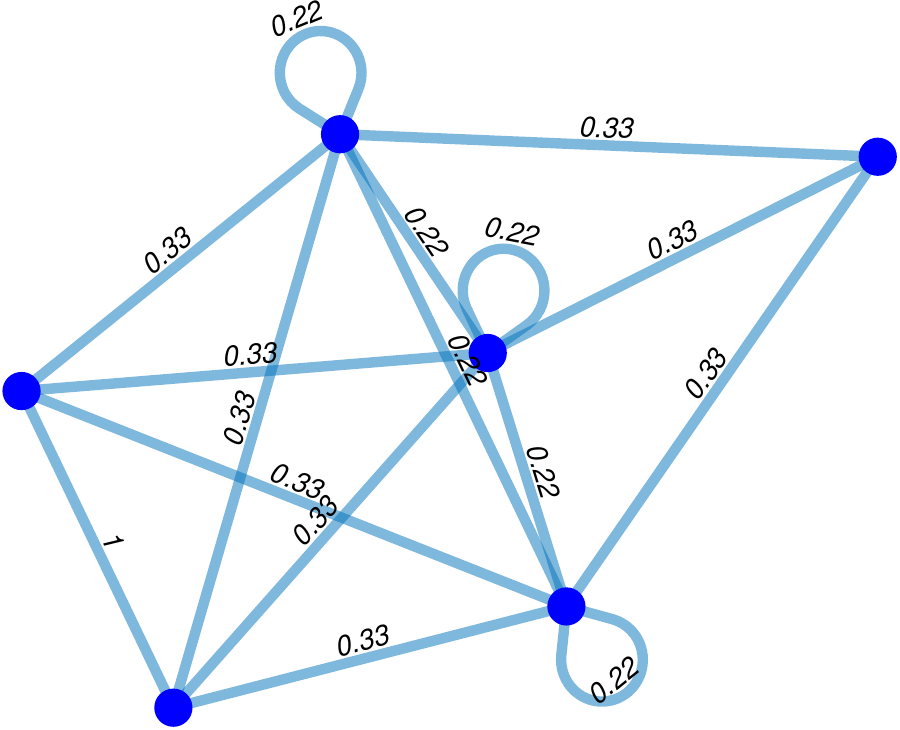}
    \quad
    \includegraphics[width=.3\textwidth]{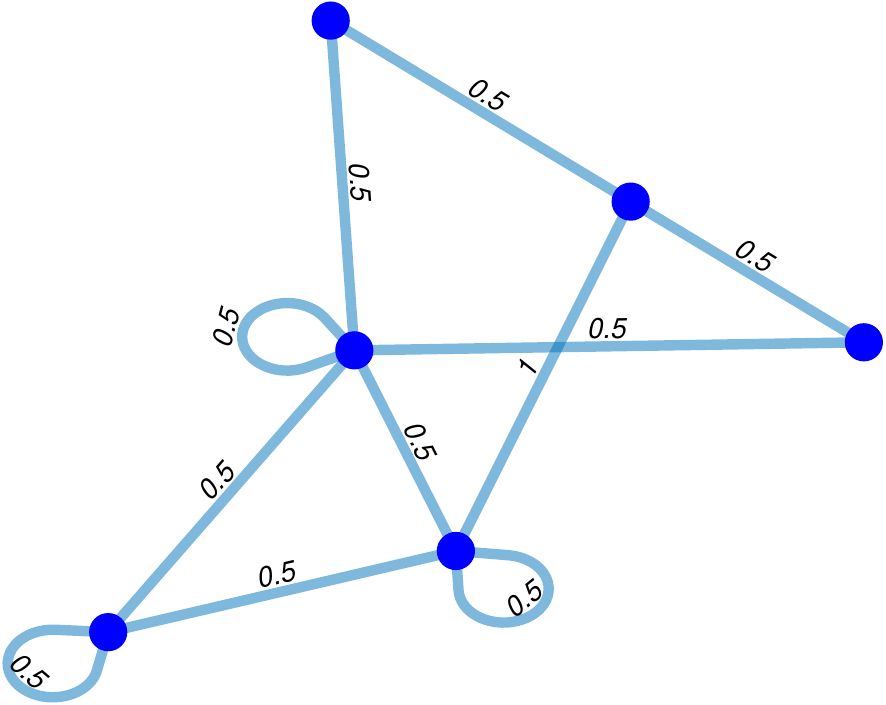}
    \quad
    \includegraphics[width=.3\textwidth]{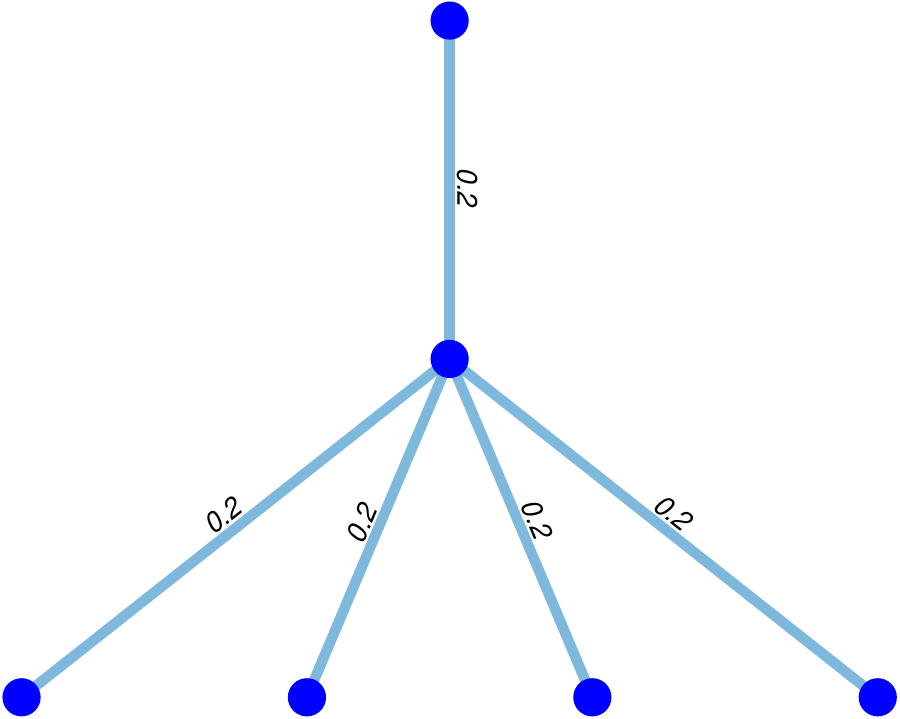}

    \caption{\small Top row: examples of positively (left), negatively (middle), positively and negatively (bipartite) pseudo-invertible graphs with singular adjacency matrix on $m=6$  vertices (top row). Corresponding weighted pseudo-inverse graphs (bottom row).}
    \label{fig-m6-pseudoinv}
\end{figure}

\begin{figure}
    \centering
    \includegraphics[width=.3\textwidth]{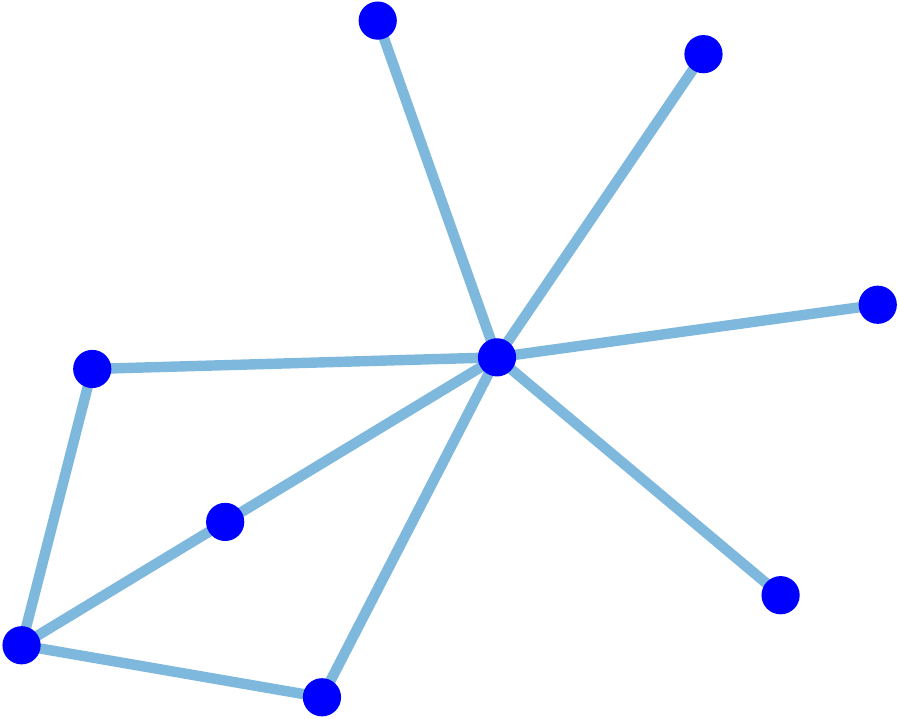}
\qquad 
    \includegraphics[width=.3\textwidth]{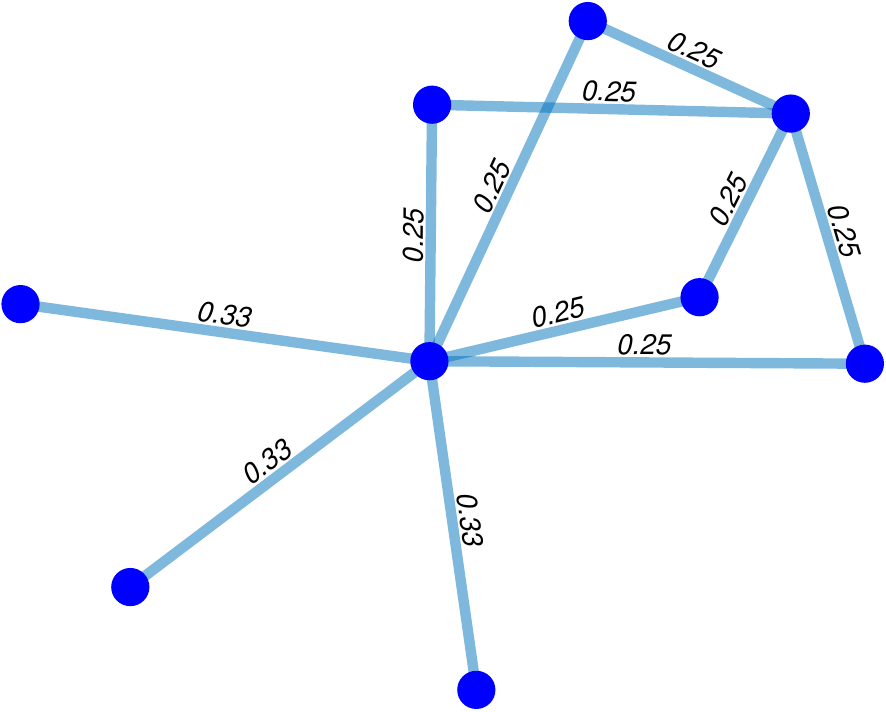}
    \caption{\small Top row: an example of a positively and negatively (bipartite) pseudo-invertible graph on $m=9$ vertices (left). Its pseudo-inverse weighted graph (right)}
    \label{fig-m9-pseudoinv}
\end{figure}

\begin{figure}
    \centering
    \includegraphics[width=.3\textwidth]{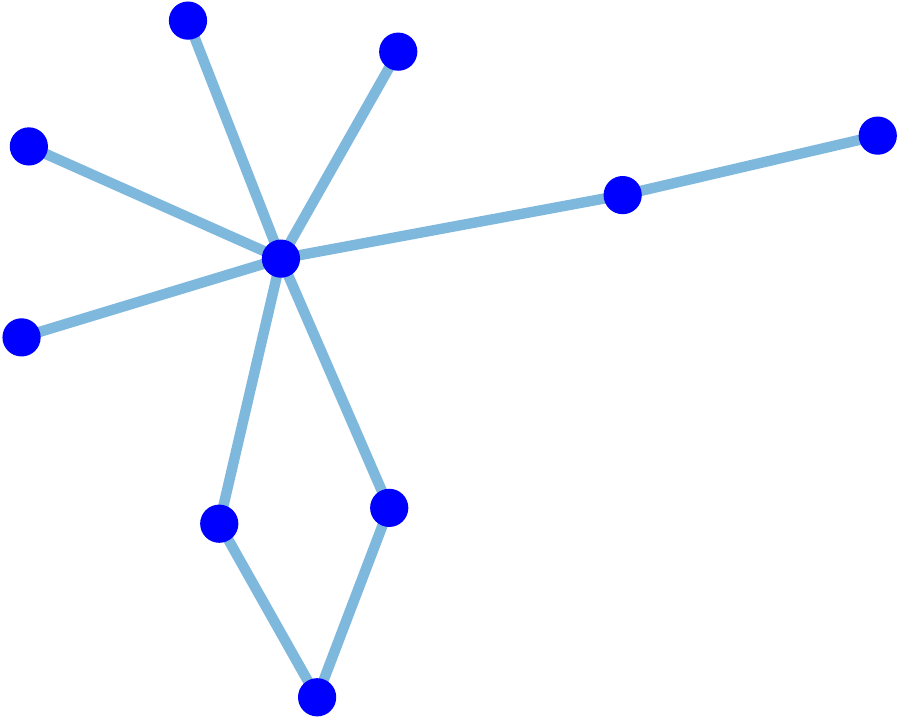}
    \quad
    \includegraphics[width=.3\textwidth]{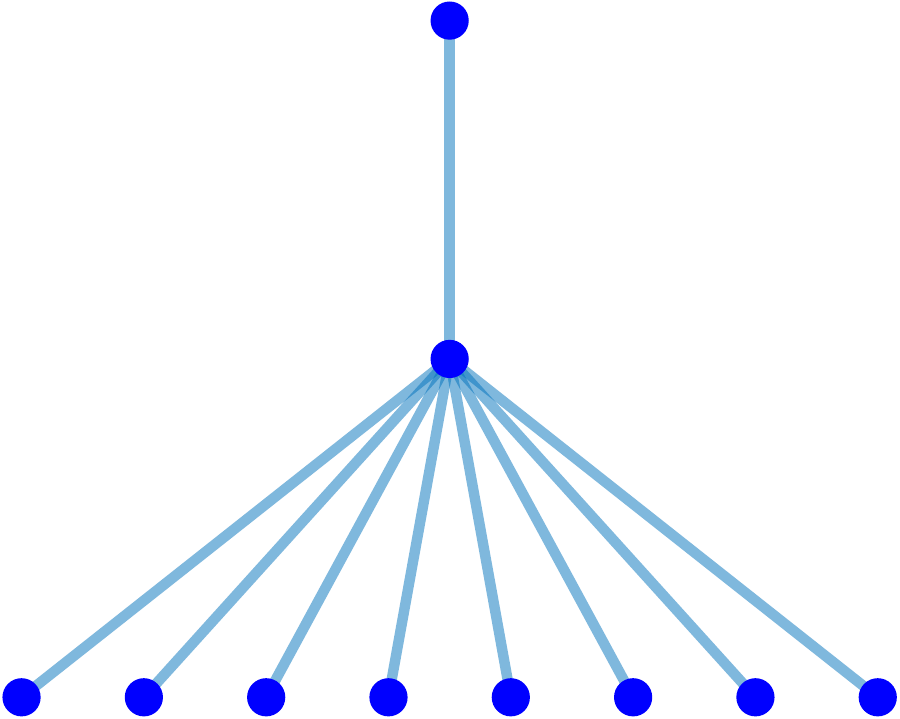}
    \caption{\small Examples of two positively and negatively (bipartite) pseudo-invertible graphs with  singular adjacency matrices on $m=10$  vertices.}
    \label{fig-m10-pseudoinv}
\end{figure}

\section{Conclusions}

In this paper we investigated the Moore-Penrose pseudo-inversion $\mathscr{M}^\dagger$ of a block symmetric matrix $\mathscr{M}$ with applications in the graph theory. We showed that maximal and minimal eigenvalues of a pseudo-inverse matrix $\mathscr{M}^\dagger$ are related to the least positive $\lambda_+(\mathscr{M})>0$ and largest negative eigenvalues $\lambda_-(\mathscr{M})<0$ of a symmetric real matrix $\mathscr{M}$. We also present the key idea of derivation of the explicit Banachiewicz–Schur form of the Moore-Penrose pseudo-inverse of a block symmetric matrix. It can be used in order to characterize the least positive and largest negative eigenvalues in terms of a mixed integer-nonlinear optimization problem. We introduced and analyzed a novel concept of positive and negative pseudo-invertibility of a symmetric matrix and graphs. We present a summary table containing number of positively/negatively pseudo-invertible graphs. The list of graphs with no more than 10 vertices contains complete information on various classes of (pseudo)invertible graphs including their spectral properties. Furthermore, we present special examples of pseudo-invertible graphs constructed from a given graph by adding hanging vertices or paths.

\end{document}